\documentclass[a4paper,USenglish,cleveref, autoref, thm-restate]{lipics-v2021}

\pdfoutput=1

\usepackage{amssymb,amsmath,graphics,graphicx}
\usepackage{pgf}
\usepackage{color,enumerate,comment,boxedminipage}
\usepackage{float}
\usepackage{cleveref}
\usepackage{comment}
\usepackage{tikz,tikz-cd}
\usepackage{xspace}
\usetikzlibrary{arrows,cd,positioning}
\usetikzlibrary{shapes,decorations}
\usetikzlibrary{calc}
\tikzset{
parallel segment below/.style={
   segment distance/.store in=\segDistance,
   segment pos/.store in=\segPos,
   segment length/.store in=\segLength,
   to path={
   ($(\tikztostart)!\segPos!(\tikztotarget)!\segLength/2!(\tikztostart)!\segDistance!90:(\tikztotarget)$) -- 
   ($(\tikztostart)!\segPos!(\tikztotarget)!\segLength/2!(\tikztotarget)!\segDistance!-90:(\tikztostart)$)  \tikztonodes
   }, 
   segment pos=.5,
   segment length=5ex,
   segment distance=2mm,
},
}

\tikzset{
parallel segment above/.style={
   segment distance/.store in=\segDistance,
   segment pos/.store in=\segPos,
   segment length/.store in=\segLength,
   to path={
   ($(\tikztostart)!\segPos!(\tikztotarget)!\segLength/2!(\tikztostart)!\segDistance!-90:(\tikztotarget)$) -- 
   ($(\tikztostart)!\segPos!(\tikztotarget)!\segLength/2!(\tikztotarget)!\segDistance!90:(\tikztostart)$)  \tikztonodes
   }, 
   segment pos=.5,
   segment length=5ex,
   segment distance=2mm,
},
}

\hideLIPIcs
\nolinenumbers

\usepackage{todonotes}
\presetkeys%
    {todonotes}%
    {inline,backgroundcolor=yellow}{}

\newtheorem*{con}{Conjecture}
\newtheorem{rthm}{Theorem}
\newenvironment{reptheorem}[1]
  {\renewcommand\therthm{\ref*{#1}}\rthm}

\newcommand{\m}[1]{vc(#1)}
\newcommand{\e}[1]{evc(#1)}
\newcommand{\sg}{single-guarded}

\newcommand{\VC}{{\sc Minimum Vertex Cover}}
\newcommand{\EVC}{{\sc Minimum Eternal Vertex Cover}}
\newcommand{\sep}{series-parallel}
\newcommand{\inte}{internal path}
\newcommand{\exte}{external path}
\newcommand{\spath}{$s$-path}
\newcommand{\tpath}{$t$-path}
\newcommand{\pe}{\mathcal{P}_{even}}
\newcommand{\po}{\mathcal{P}_{odd}}
\newcommand{\cf}[2]{U_{#1,#2}}
\newcommand{\al}{alt}

\newcommand\restr[2]{{
  \left.\kern-\nulldelimiterspace
  \littletaller
  \right|_{#2}
  }}

\newcommand{\littletaller}{\mathchoice{\vphantom{\big|}}{}{}{}}

\title{The Minimum Eternal Vertex Cover Problem on a Subclass of Series-Parallel Graphs}
\titlerunning{Minimum Eternal Vertex Cover on Series-Parallel Graphs}

\author{Tiziana {Calamoneri}}{Sapienza University of Rome, Department of Informatics, Italy \and \url{https://sites.google.com/di.uniroma1.it/tiziana-calamoneri/home-page}}{calamo@di.uniroma1.it}{https://orcid.org/0000-0002-4099-1836}{}
\author{Federico {Corò}}{University of Padua, Mathematics Department, Italy \and \url{https://sites.google.com/view/federicocoro/home?authuser=0}}{Fedo.coro@gmail.com}{https://orcid.org/0000-0002-7321-3467}{}
\author{Giacomo {Paesani}}{Sapienza University of Rome, Department of Informatics, Italy \and \url{https://giacomopaesani.github.io/}}{paesani@di.uniroma1.it}{https://orcid.org/0000-0002-2383-1339}{}
\authorrunning{T. Calamoneri et al.}

\Copyright{Tiziana Calamoneri and Federico Corò and Giacomo Paesani}



\EventEditors{}
\EventLongTitle{50th International Symposium on Mathematical Foundations of Computer Science (MFCS 2025)}
\EventShortTitle{MFCS 2025}
\EventAcronym{MFCS}
\EventYear{2025}
\EventDate{August 25--29, 2025}
\EventLocation{Warsaw, Poland}
\EventLogo{}
\SeriesVolume{}
\ArticleNo{}

\begin{document}

\maketitle

\begin{abstract}
{\em Eternal vertex cover} is the following two-player game between a defender and an attacker on a graph. Initially, the defender positions $k$ guards on $k$ vertices of the graph; the game then proceeds in turns between the defender and the attacker, with the attacker selecting an edge and the defender responding to the attack by moving some of the guards along the edges, including the attacked one. The defender wins a game on a graph $G$ with $k$ guards if they have a strategy such that, in every round of the game, the vertices occupied by the guards form a vertex cover of $G$, and the attacker wins otherwise. The {\em eternal vertex cover number} of a graph $G$ is the smallest number $k$ of guards allowing the defender to win and \EVC\ is the problem of computing the eternal vertex cover number of the given graph.

We study this problem when restricted to the well-known class of \sep\ graphs. In particular, we prove that \EVC\ can be solved in linear time when restricted to melon graphs, a proper subclass of series-parallel graphs. Moreover, we also conjecture that this problem is NP-hard on \sep\ graphs.
\end{abstract}

\section{Introduction}

A vertex cover of a graph $G=(V, E)$ is a set $S\subseteq V$ such that, for every edge in $E$, at least one of its endpoints is in $S$. A minimum vertex cover of $G$ is a vertex cover of $G$ of minimum cardinality. This minimum value, denoted by $\m{G}$, is called the vertex cover number of $G$. The \VC\ problem consists in determining this number.

The notion of {\em eternal vertex cover}, first introduced by Klostermeyer and Mynhardt~\cite{KlostermeyerM09}, exploits the above definition in the context of a two-player multi-round game, where a {\em defender} uses mobile guards placed on some vertices of $G$ to protect the edges of $G$ from an {\em attacker}. The game begins with the defender placing guards on some vertices, at most one {\em per} vertex. The total number of guards remains the same throughout the game. In each round of the game, the attacker chooses an edge to attack. In response, the defender moves the guards so that each guard either stays at its current location or moves to an adjacent vertex; the movement of all guards in a round is assumed to happen in parallel. If a guard crosses the attacked edge during this move, it {\em protects} the edge from the attack. The defender wins if the edges can be protected by any sequence of attacks. If an attacked edge cannot be protected in some round, the attacker wins. It is easy to see that a necessary condition to protect the graph is that the set of vertices where the guards lie is a vertex cover, and this justifies the name of eternal vertex cover.

The \EVC\ problem consists in determining the {\em eternal vertex cover number} of $G$, denoted by $\e{G}$, that is, the minimum number of guards allowing the defender to protect all the edges of $G$. In the literature, $\e{G}$ is sometimes denoted by $\alpha^\infty_m(G)$ (see for example~\cite{KlostermeyerM09}) or by $\tau^{\infty}(G)$~\cite{AFI15}.

The \EVC\ finds applications in network security, drone surveillance, and war scenarios. For example, some agents are deployed on the nodes of a network in such a way that the agents watch every connection between nodes. A malicious attack forces an agent to traverse that connection and, more in general, to reconfigure the position of the agents. The eternal vertex cover game asks whether it is possible for a set of agents to respond to any sequence of attacks. Minimizing the number of agents required for an everlasting defense and understanding a winning strategy is clearly beneficial to resource allocation. 

A {\em Series-parallel graph} can be recursively constructed by observing that a single edge is a series-parallel graph, and by composing smaller series-parallel graphs either in {\em series} or in {\em parallel}. Although this class has been introduced a long time ago~\cite{Du65}, it still attracts the attention of researchers (see, {\em e.g.},~\cite{AN20,ASW23,BMM24,DissauxDNN23,Ma19}). Series-parallel graphs are a well-known and studied graph class from a theoretical perspective and naturally model two-terminal networks that are constructed with the series and parallel composition. In this case, the total values of the fundamental parameters can be computed directly.

In this paper, we study the \EVC\ problem when restricted to series-parallel graphs: we prove that it can be linearly solved for a proper subclass of series-parallel graphs while we conjecture that it remains NP-hard on the whole class.

In the following, we survey the existing literature both on eternal vertex cover and on \sep\ graphs, and then we describe in detail our results.

\subsection{Previous Results}

Since its definition, the \EVC\ problem has been deeply studied from a computational complexity point of view: deciding whether \(k\) guards can protect all the edges of a graph is NP-hard~\cite{FominGGKS10}; it remains so even for bipartite graphs~\cite{MisraN22} and for biconnected internally triangulated planar graphs, although there exists a polynomial time approximation scheme for computing the eternal vertex cover number on this class of graphs~\cite{BabuCFPRW22}. The problem can be exactly solved in $2^{\mathcal{O}(n)}$ time and is FPT parameterized by solution size~\cite{FominGGKS10}. 

On the positive side, there are a few graph classes for which the problem can be efficiently solved. Indeed, it is solvable in linear time on trees and cycles~\cite{KlostermeyerM09}, maximal outerplanar graphs~\cite{BKPW22}, chain and split graphs~\cite{PP24}. Moreover, it is solvable in quadratic time on chordal graphs~\cite{BabuCFPRW22,BabuP22} and  solvable in polynomial time on co-bipartite graphs~\cite{BabuMN22}, cographs~\cite{PP24} and generalized trees~\cite{AFI15}.

The vertex cover and the eternal vertex cover number are linearly equivalent parameters (see for example~\cite{BMPY23} for a formal definition of linear equivalent parameters), and it holds that $\m{G}\leq \e{G}\leq 2\m{G}$~\cite{KlostermeyerM09}. Consequently, it is also interesting to understand for which graphs this relation holds that these two parameters are very close: the authors of~\cite{KlostermeyerM09,KMC16} show different conditions for equality (graphs for which this relation holds are generally called {\em spartan}), while in~\cite{BabuCFPRW22}, it is showed that \(\e{G} \le \m{G} + 1\) for every locally connected graph $G$.

One of the reasons of interest for \sep\ graphs is that many combinatorial problems that are computationally hard on general graphs become polynomial-time or even linear-time solvable when restricted to the \sep\ graphs ({\em e.g.}, vertex cover~\cite{tollis89}, dominating set~\cite{Va12}, coloring~\cite{AP89}, graph isomorphism~\cite{GS15,LPPS14} and Hamiltonian cycle~\cite{EGW01,GHO11}).

On the other hand, very few problems are known to be NP-hard for \sep\ graphs. These include the subgraph isomorphism~\cite{DLP96,GN96,MT92}, the bandwidth~\cite{Sy83}, the edge-disjoint paths~\cite{ZTN00}, the common subgraph~\cite{KKM18} and the list edge and list total coloring~\cite{ZMN05} problems.

\subsection{Our Results}\label{sec:results}

In this work, we study \EVC\ on the class of series-parallel graphs. First, we consider the subclass of \emph{melon graphs}, which is constituted by a set of pairwise internally disjoint paths linking two vertices, and in~\Cref{sec:melon}, the core of the paper, we prove the following result:

\begin{theorem}\label{thm:melons}
\EVC~is linear-time solvable for melon graphs.
\end{theorem}

\medskip

The proof of the aforementioned result is based on a case-by-case analysis classifying melon graphs according to the number of paths of even and odd lengths. For each possible input melon graph, we not only compute the eternal vertex cover number in linear time, but we also provide a minimum eternal vertex cover class and defense strategies.

In~\Cref{sec:sp}, we extend our analysis to the whole class of \sep\ graphs and propose the following:

\begin{con}\label{con:sp}
\EVC\ is NP-hard on series-parallel graphs.
\end{con}

We formalize some arguments supporting this conjecture: we gather evidence that the class of melon graphs is substantially small when compared to \sep\ graphs with respect to a number of structural and algorithmic properties.

\section{Terminology}
For a positive integer $k$, we denote with $[k]$ the set $\{1,\ldots, k\}$. 

Let $G=(V, E)$ be a graph, on which we recall the following definitions. Given a vertex $v$ of $G$, the {\em closed neighborhood of} $v$ is the set of vertices that are adjacent to $v$ and $v$ itself, and it is denoted by $N[v]$. A {\em path} $P=(V(P),E(P))$ is a graph, where $V(P)$ is $\{v_0, \ldots , v_\ell\}$, $\ell \geq 1$, and $E(P)$ is $\{ v_iv_{i+1}~|~i=0, \ldots, \ell-1\}$; $\ell$ is the {\em length} of $P$.

A graph $G=(V, E)$ is {\em bipartite} if it is possible to partition the vertex set into two not empty subsets: $V=A \cup B$ so that each edge of $E$ can only connect one vertex in $A$ with one vertex in $B$; in this case, we represent $G$ with $(A\cup B, E)$. For extended graph terminology, we refer to \cite{D12}.

\subsection{Eternal Vertex Cover}

Given a graph $G=(V, E)$ and a subset of vertices $U\subseteq V$, we imagine each vertex of $U$ hosting one guard, and all the edges incident to these vertices are considered {\em guarded}. The guards are allowed to move from one vertex to another only through an edge connecting them. 

An {\em attack} is the selection of one edge $e\in E$ by the attacker. The defender {\em protects} an attacked edge if it can move a guard along that edge. Thus, it is possible only to protect guarded edges and a necessary condition for $U\subseteq V$ to be able to protect any edge from an attack is that $U$ is a vertex cover of $G$.

Consider a guarded edge $e=vw$ and, without loss of generality, assume that $v\in U$. A {\em defense} from the attack on $e$ is defined as a one-to-one function $\phi: U \rightarrow V$ such that $e$ is protected, that is $\phi(v) = w$, and for each $u \in U$, $\phi(u) \in N[u]$. Given any vertex $u \in U$, we say that the guard on $u$ {\em  shifts to $\phi(u)$} and, by extension, $U$ {\em shifts to} $U'$ where $U'=\phi(U)=\{ \phi(u) \,\,\mbox{ s.t. } u \in U\}$. 

The protection of an attacked edge $vw$ with a guard on both endpoints can be easily guaranteed by shifting the guard on $v$ to $w$, the guard on $w$ to $v$, and every other guard stays on the same vertex. So, in the following, we implicitly assume that an attack always happens on an edge guarded by one guard, and called {\em \sg} edge.

We are now ready to give the notion of eternal vertex cover.

\begin{definition} \cite{BabuCFPRW22}
Given a graph $G$, a family $\mathcal{U}$ of vertex covers of $G$ all of the same cardinality is an {\em eternal vertex cover class} of $G$ if the defender can protect any attacked edge by shifting any vertex cover of $\mathcal{U}$ to another vertex cover of $\mathcal{U}$. Each vertex cover of $\mathcal{U}$ is called a {\em configuration} for $G$. The {\em size} of an eternal vertex cover class $\mathcal{U}$ is the cardinality of any configuration of $\mathcal{U}$. The \EVC\ problem consists of finding the minimum size of an eternal vertex cover class for $G$, and this number is denoted with $\e{G}$. An eternal vertex cover class of size $\e{G}$ is said to be a {\em minimum eternal vertex cover class}.
\end{definition}

In the following, in order to determine $\e{G}$, first we provide a family $\mathcal{U}$ of vertex covers; then, for every vertex cover $U$ of $\mathcal{U}$ and every edge $e$ of $G$, we exhibit a defense function that shifts $U$ to another vertex cover of $\mathcal{U}$ and protects $e$, thus showing that $\mathcal{U}$ is an eternal vertex cover of $G$; finally, we show that no eternal vertex cover class of $G$ can have size strictly smaller than $\mathcal{U}$.

\subsection{Series-Parallel Graphs}

Let the graphs considered from now on have two distinguished vertices, $s$ and $t$, called {\em source} and {\em sink}, respectively.

Let be given two vertex-disjoint graphs $G_1$ and $G_2$, with sources and sinks $s_1$ and $t_1$, $s_2$ and $t_2$, respectively. The {\em series composition} of $G_1$ and $G_2$ is a graph $G$ obtained by merging $t_1$ with $s_2$, and its distinguished vertices are $s=s_1$ and $t=t_2$. The {\em parallel decomposition} of $G_1$ and $G_2$ is a graph $G$ obtained by merging $s_1$ with $s_2$ into distinguished vertex $s$ and $t_1$ with $t_2$ into distinguished vertex $t$.

Series-parallel graphs can be constructed recursively by series and parallel compositions:

\begin{definition}\cite{Du65}
A {\em series-parallel} graph $G$ is a graph with two distinguished vertices $s$ and $t$ that is either a single edge or can be recursively constructed by either {\em series} or {\em parallel} composition of two series-parallel graphs. 
\end{definition}

Due to the recursive nature of \sep\ graphs, it is natural to introduce a decomposition that mimics the construction of these graphs.

\begin{definition}\cite{ValdesTL82}
The {\em SP-decomposition tree} of a series-parallel graph $G$ is a rooted binary tree $T$ in which each leaf corresponds to an edge of $G$, and every internal node of $T$ is labeled as either a parallel or series node; starting from its edges, that are series-parallel graphs, the series-parallel subgraph associated to a subtree of $T$ rooted at a node $v$ is the composition indicated by the label of $v$ of the two series-parallel subgraphs associated to the children of $v$; $G$ is the series-parallel graph associated to the root of $T$.
\end{definition}

For an extended and more formal treatment of \sep\ graphs and SP-decompositions, the reader can refer {\em e.g.} to~\cite{DissauxDNN23}.

\subsection{Melon Graphs}

The main result of this paper, described and proved in~\Cref{sec:melon}, deals with a subclass of \sep\ graphs:

\begin{definition}
For any integer $k\geq 1$, given $k$ internally vertex-disjoint paths $P^{(1)}, \ldots , P^{(k)}$ whose extremes are their distinguished vertices, a graph $G$ is a {\em $k$-melon graph} if  $G$ can be constructed by the parallel composition of $P^{(1)}, \ldots , P^{(k)}$. A graph $G$ is a {\em melon graph} if it is a~$k$-melon graph, for some $k\geq 1$.
\end{definition}

In particular, paths are 1-melon graphs and cycles are 2-melon graphs. Note that for every $k \neq 2$, in every $k$-melon graph $G$, $s$ and $t$ are the only two vertices of $G$ not having degree two.

Melon graphs have already been studied in different research works: with respect to the computation of the treelength~\cite{DissauxDNN23}, for the understanding of the treewidth on hereditary graph classes~\cite{AbrishamiCV22,SintiariT21} and in high-energy physics representing tensor models~\cite{BGRR11}.

Let $G$ be a $k$-melon graph for some $k\geq 1$,  constituted by paths $P^{(1)}, \ldots , P^{(k)}$. Denote with $\mathcal{P}(G)$ (or simply $\mathcal{P}$ if there is no risk for confusion) the set of paths $P^{(1)}, \ldots, P^{(k)}$ used to obtain $G$. A path is said to be either an {\em odd} or an {\em even path}, depending on the parity of its length. Let $\mathcal{P}=\po \cup \pe$ be the partition of $\mathcal{P}$ into odd and even paths.

\begin{definition}
A $k$-melon graph $G$ obtained by the paths of $\mathcal{P}=\pe \cup \po$ is an {\em even} (respectively {\em odd}) {\em $k$-melon graph} if $\po=\emptyset$ (respectively $\pe=\emptyset$), and it is {\em mixed} otherwise. 
\end{definition}

In what follows, we indicate by $P_e$ a path in $\pe$ and by $P_o$ a path in $\po$, in order to easily have in mind its parity when confusion may arise.

\section{Eternal Vertex Cover on Melon Graphs}\label{sec:melon}

In this section, we provide the eternal vertex cover number of melon graphs, and our proofs are constructive. More in detail, the main result of this paper is the following:

\begin{reptheorem}{thm:melons}
\EVC~is linear-time solvable for melon graphs.
\end{reptheorem}

\medskip

In the following, we will prove~\Cref{thm:melons} separately on even, odd, and mixed melon graphs. Note that it is very well-known how to solve \EVC~on $1$- and $2$-melon graphs, {\em i.e.}, paths and cycles~\cite{KlostermeyerM09}; hence, in the rest of this work, we only consider $k$-melon graphs with $k\geq 3$.

\subsection{Odd Melon Graphs}\label{ss:odd}

In order to prove~\Cref{thm:melons} on odd melon graphs, we exploit a result from~\cite{MisraN23}, for which we need some additional definitions.

A {\em matching} $M$ of $G$ is a subset of vertex-disjoint edges of $G$. Moreover, if $G$ is bipartite and  $V=A \cup B$, a matching $M$ is {\em perfect} if $|M|=\min \{ |A|, |B| \}$; clearly, if $|A|=|B|$, every vertex is adjacent to some edge of a perfect matching.

Given an odd path $P$ of length $\ell$, we can recognize on it a maximum matching of cardinality $(\ell+1)/2$ and a maximal matching of cardinality $(\ell-1)/2$; the first one is perfect, and hence we call it {\em odd-perfect}, while the second leaves the two endpoints of the path out of the matching, and so we denote it as {\em odd-imperfect}. It is easy to see that every edge of $P$ belongs to exactly one of these two matchings.

In support of our goal of building constructive proofs, we say that a bipartite graph $G$ is {\em elementary} if it is connected and every edge belongs to some perfect matching of $G$~\cite{Ga64}. 

The following result connects elementary graphs and their eternal vertex cover number:

\begin{lemma}\cite{MisraN23}\label{lem:mini}
Let $G$ be an elementary graph, then $\e{G}=\m{G}=|V(G)|/2$.
\end{lemma}

We exploit the previous lemma to prove our results on odd melon graphs. Preliminarily, observe that every odd melon graph $G$ is bipartite, so for the rest of this subsection, we assume that $G=(A\cup B,E)$. Since every path has an odd length, then one between $s$ and $t$ belongs to $A$ while the other belongs to $B$; without loss of generality, we assume $s\in A$ and $t\in B$.

\begin{lemma}\label{lem:elem}
Every odd melon graph is elementary.    
\end{lemma}

\begin{proof}
Every melon graph is connected by definition, so it remains to prove that any edge $e$ of $G$ belongs to a perfect matching $M_e$ that we construct as follows. 

For each path of $\mathcal{P}(G)$, consider its odd-perfect and odd-imperfect matchings. Without loss of generality, let $e \in P^{(1)}$ (otherwise we can rename the paths in $\mathcal{P}$). If $e$ belongs to the odd-perfect matching of $P^{(1)}$ (see the red edge in~\Cref{fig:cboe}.a),  then put in $M_e$ all the edges of this odd-perfect matching (including $e$) and all the edges lying in the odd-imperfect matchings of all the other paths. If, vice versa, $e$ belongs to the odd-imperfect matching of $P^{(1)}$ (see the red edge in~\Cref{fig:cboe}.b), then put in $M_e$ all the edges of the odd-perfect matching of $P^{(2)}$ and all the edges lying on the odd-imperfect matchings of all the other paths (including $e$). 

$M_e$ contains $e$ and is a perfect matching indeed, due to the alternating nature of $M_e$, for every vertex $v$ of $G$, there exists exactly one edge of $M_e$ that contains $v$.
\end{proof}

Note that each odd melon is bipartite, and it holds that $|A|=|B|$ because, for any path $P \in \mathcal{P}$, $|A \cap P|=|B \cap P|$. Moreover, $A$ and $B$ are two vertex covers of $G$. This observation is exploited to prove the following result.

\begin{theorem}\label{lem:odd}
Let $G=(A\cup B,E)$ be an odd $k$-melon graph. It holds that $\e{G}=\m{G}$, and the family $\mathcal{U}=\{A, B\}$ is a minimum eternal vertex cover class of $G$.
\end{theorem}

\begin{proof}
Consider an edge $e$ of $G$. Since $G$ is elementary by~\Cref{lem:elem}, there exists a perfect matching $M_e$ of $G$ that contains $e$, and $M_e$ can be found following the proof of~\Cref{lem:elem}. 

Whenever attacked, the edge $e$ can always be protected. Indeed, suppose first that the guards are positioned on the vertices of $A$; then, to protect $e$, it is enough that every guard shifts through its incident edge in $M_e$, {\em i.e.}, for each $a\in A$, $\phi(a)=b$, where $ab$ is the unique edge of $M_e$ incident to $a$. The case in which the guards are positioned on the vertices of $B$ is done symmetrically.
\end{proof}

\subsection{Even Melon Graphs}

Let $G$ be an even melon graph. Although it is easy to see that $G$ is bipartite, we can not exploit a strategy similar to the proof of~\Cref{lem:odd} because for an even $k$-melon graph it holds that the two bipartitions have the same cardinality if and only if $k=2$. Hence, we follow another approach that needs some further definitions.

Let $G$ be an even melon graph and $U \subseteq V$ be a subset of vertices. Let $P$ be a path in $\mathcal{P}$; in view of its parity, let its length equal to ${2m}$, for some $m\geq 1$.

We distinguish the two following behaviors of $P$ with respect to $U$: we say that $P$ is an {\em \inte} with respect to $U$ (or simply an \inte, if $U$ is clear from the context) if $U \cap V(P)=\{v_{2j}~|~j\in \{0,\ldots,m\}\}$ and similarly, that $P$ is an {\em \exte} with respect to $U$ (or simply an \exte, if $U$ is clear from the context) if $U\cap V(P)=\{v_{2j+1}~|~j\in \{0,\ldots,m-1\}\}\cup \{s,t\}$. Intuitively, $s$ and $t$ belong both to internal and to external paths; moreover, the inner vertices of an \exte~alternately belong to $U$, starting with the neighbors of $s$ and of $t$, while the neighbors of $s$ and of $t$ do not belong to $U$ in an \inte. As an example, in~\Cref{fig:cboe}.c, the three leftmost paths and the rightmost one are internal, while the remaining one is external. 

\begin{lemma}\label{le:ciclo_pari}
Let $G$ be an even $2$-melon graph with paths $P$ and $P'$, source $s$ and sink $t$. Moreover, let $U$ be a set of vertices such that $P$ is internal and $P'$ is external with respect to $U$, and let $U'$ be a set of vertices such that $P'$ is internal and $P$ is external with respect to $U'$. Then it is possible to defend $G$ from an attack on any \sg~edge by shifting $U$ to $U'$ and vice versa.
\end{lemma}

\begin{proof}[Proof of~\Cref{le:ciclo_pari}]
Let $e=zw$ be an edge of $G$. Intuitively, to protect $e$, we move the guards to turn $P$ into an \exte~and $P'$ into an \inte~following the direction of the forced shift of the guard on $e$. Let $e=zw$ be an edge of $G$. Since $U$ is a vertex cover and $e$ is \sg, it is not restrictive to assume that $z\in U$ and $w\notin U$. Call $u_0,\ldots,u_{2m}$ the vertices of $P$ and $v_0,\ldots,v_{2m'}$ the vertices of $P'$, for some $m, m'\geq 1$, and let $u_0=v_0=t$ and $u_{2m}=v_{2m'}=s$. Then, to protect $e$, we move the guards to turn $P$ into an \exte~and $P'$ into an \inte~following the forced shift of the guard from $z$ to $w$.

In particular, assume that $e$ is either an edge of $P$ and $z=u_{2j}$ and $w=u_{2j+1}$ for some $0 \leq j<m$, or an edge of $P'$ and $z=v_{2j+1}$ and $w=v_{2j}$ for some $0 \leq j<m'$. Then, to protect $e$, we use the following defense function $\phi$:
\begin{itemize}
\item $\phi(u_{2i})=u_{2i+1}$ for $i=0, \ldots , m-1$;
\item $\phi(v_{2i+1})=v_{2i}$ for $i=0, \ldots , m'-1$;
\item $\phi(s)=s$.
\end{itemize}
It is clear that $z$ shifts to $w$ and $C$ to $C'$. Due to symmetry, a similar defense function defends the attack of $e$ when it is either an edge of $P$ and $z=u_{2j}$ and $w=u_{2j-1}$ for some $0 < j\leq m$, or an edge of $P'$ and $z=v_{2j-1}$ and $w=v_{2j}$ for some $0 <j<m'$. 
\end{proof}

Now, given a $k$-even melon graph, for each fixed $i\in [k]$, we denote with $U_i$ the vertex set such that the path $P^{(i)}$ is an \exte~w.r.t. $U_i$ and the path $P^{(j)}$ is an \inte~w.r.t. $U_i$, for every $j\in [k]$ and $j\neq i$. In the following theorem, we exploit~\Cref{le:ciclo_pari} to defend any even $k$-melon with $k \geq 3$ with its guards on the vertices of $U_i$ by considering the even 2-melon graph induced by $P^{(i)}$ and one of the \inte s w.r.t. $U_i$.

\begin{theorem}\label{lem:evenmelon}
Let $G$ be an even $k$-melon graph, for some $k\geq 3$. It holds that $\e{G}=\m{G} +1$, and the family $\mathcal{U}=\{U_i~|~i\in [k]\}$ is a minimum eternal vertex cover class of $G$, where the sets $U_i$ are defined above.
\end{theorem}

\begin{proof}
First, observe that, fixed any $i\in [k]$, the set $U_i$ is a vertex cover of $G$ with $\m{G}+1$~elements. Indeed, due to the alternating nature of the definition, every edge of $G$ contains exactly one vertex of $U_i$ with the exception of the two edges of the external path $P^{(i)}$ which are incident to $s$ and $t$, whose both endpoints are vertices of $U_i$. 

Consider now the set $U$ of vertices of $G$ such that every path $P\in \mathcal{P}$ is internal w.r.t. $U$. Clearly, since no \exte s are in $U$, it holds that $|U_i|=|U|+1$, for every $i\in [k]$. Moreover, $U$ is a vertex cover and it is of minimum cardinality because every edge is incident to exactly one vertex in $U$. Finally, $U$ is the unique minimum vertex cover of $G$, and so it cannot be a configuration of a minimum eternal vertex cover class. It follows that $\e{G}$ is at least $\m{G}+1$. Then, proving that $\mathcal{U}$ is an eternal vertex cover class of $G$ also shows that $\mathcal{U}$ is minimum.

Let $U_i$ be any configuration of $\mathcal{U}$ and let $e$ be the attacked edge of $G$. Let $P^{(j)}$ be the path which contains $e$. If $j=i$ (see~\Cref{fig:cboe}.c), let $P^{(k)}$ be any internal path of $\mathcal{P}$ w.r.t. $U_i$ and let $G'$ be the subgraph of $G$ induced by the vertices of $P^{(i)}$ and $P^{(k)}$. If $j\neq i$ (see~\Cref{fig:cboe}.d), let $G'$ be the subgraph of $G$ induced by the vertices of $P^{(i)}$ and $P^{(j)}$. Observe that $G'$ is an even $2$-melon graph; calling $\phi'$ the defense function of~\Cref{le:ciclo_pari} to defend $G'$ from the attack on $e$, to protect $e$ in $G$ we define the defense function $\phi$ as follows: $\phi(v)=\phi'(v)$ if $v$ is a vertex of $G'$ and $\phi(v)=v$ otherwise. It is easy to see that $\phi$ protects $e$.
\end{proof}

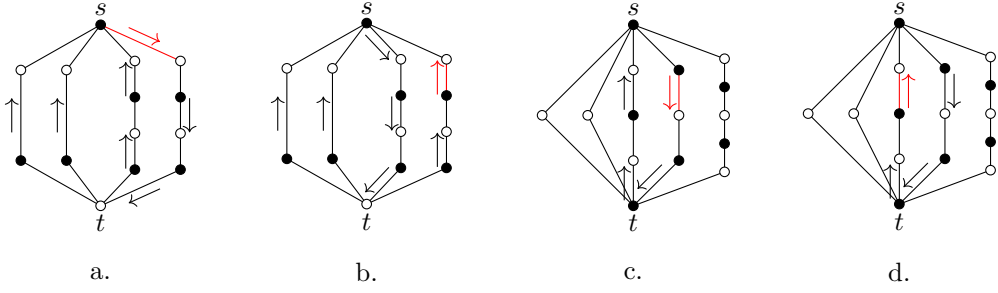
\begin{figure}
\centering
\begin{minipage}{0.25\textwidth}
\begin{tikzpicture}[scale=0.6]
\coordinate (S) at (0,2);\coordinate (T) at (0,-2);\coordinate (A1) at (-1.75,1);\coordinate (A2) at (-1.75,-1);\coordinate (B1) at (-0.75,1);\coordinate (B2) at (-0.75,-1);\coordinate (C1) at (0.75,1.2);\coordinate (C2) at (0.75,0.4);\coordinate (C3) at (0.75,-0.4);\coordinate (C4) at (0.75,-1.2);\coordinate (D1) at (1.75,1.2);\coordinate (D2) at (1.75,0.4);\coordinate (D3) at (1.75,-0.4);\coordinate (D4) at (1.75,-1.2);
\draw[->,color=red](S)to[parallel segment below](D1);\draw[color=red] (S)--(D1);\draw[->](D2)to[parallel segment below](D3);\draw[->](D4)to[parallel segment below](T);\draw[->](C2)to[parallel segment below](C1);\draw[->](C4)to[parallel segment below](C3);\draw[->](A2)to[parallel segment below](A1);\draw[->](B2)to[parallel segment below](B1);
\draw(D2)--(D3)(D4)--(T)(C2)--(C1)(C4)--(C3)(A2)--(A1)(B2)--(B1)(S)--(A1)(A2)--(T)(S)--(B1)(B2)--(T)(S)--(C1)(C2)--(C3)(C4)--(T)(D1)--(D2)(D3)--(D4);
\draw[fill=white] (T) circle [radius=3pt](A1) circle [radius=3pt](B1) circle [radius=3pt](C1) circle [radius=3pt](C3) circle [radius=3pt](D1) circle [radius=3pt](D3) circle [radius=3pt];
\draw[fill=black](S) circle [radius=3pt](A2) circle [radius=3pt](B2) circle [radius=3pt](C2) circle [radius=3pt](D2) circle [radius=3pt](C4) circle [radius=3pt](D4) circle [radius=3pt];
\node[below] at (T) {$t$};\node[above] at (S) {$s$};\node at (0, -3.5) {a.};
\end{tikzpicture}
\end{minipage}%
\begin{minipage}{0.25\textwidth}
\begin{tikzpicture}[scale=0.6]
\coordinate (S) at (0,2);\coordinate (T) at (0,-2);\coordinate (A1) at (-1.75,1);\coordinate (A2) at (-1.75,-1);\coordinate (B1) at (-0.75,1);\coordinate (B2) at (-0.75,-1);\coordinate (C1) at (0.75,1.2);\coordinate (C2) at (0.75,0.4);\coordinate (C3) at (0.75,-0.4);\coordinate (C4) at(0.75,-1.2);\coordinate (D1) at (1.75,1.2);\coordinate (D2) at (1.75,0.4);\coordinate (D3) at (1.75,-0.4);\coordinate (D4) at (1.75,-1.2);
\draw[->,color=red](D2)to[parallel segment below](D1);\draw[color=red] (D2)--(D1);\draw[->](D4)to[parallel segment below](D3);\draw[->](C4)to[parallel segment above](T);\draw[->](C2)to[parallel segment above](C3);\draw[->](S)to[parallel segment above](C1);\draw[->](B2)to[parallel segment below](B1);\draw[->](A2)to[parallel segment below](A1);
\draw[color=red] (D1)--(D2);\draw(D3)--(D4)(S)--(C1)(C2)--(C3)(C4)--(T)(A2)--(A1)(B2)--(B1)(S)--(A1)(A2)--(T)(S)--(B1)(B2)--(T)(S)--(D1)(D2)--(D3)(D4)--(T)(C1)--(C2)(C3)--(C4);
\draw[fill=white] (T) circle [radius=3pt](A1) circle [radius=3pt](B1) circle [radius=3pt](C1) circle [radius=3pt](C3) circle [radius=3pt](D1) circle [radius=3pt](D3) circle [radius=3pt];
\draw[fill=black] (S) circle [radius=3pt](A2) circle [radius=3pt](B2) circle [radius=3pt](C2) circle [radius=3pt](D2) circle [radius=3pt](C4) circle [radius=3pt](D4) circle [radius=3pt];
\node[below] at (T) {$t$};\node[above] at (S) {$s$};\node at (0, -3.5) {b.};
\end{tikzpicture}
\end{minipage}%
\begin{minipage}{0.25\textwidth}
\begin{tikzpicture}[scale=0.6]
\coordinate (S) at (0,2);\coordinate (T) at (0,-2);\coordinate (A1) at (-2,0);\coordinate (B1) at (-1,0);\coordinate (C1) at (0,1);\coordinate (C2) at (0,0);\coordinate (C3) at (0,-1);\coordinate (D1) at (1,1);\coordinate (D2) at (1,0);\coordinate (D3) at (1,-1);\coordinate (E1) at (2,1.25);\coordinate (E2) at (2,0.625);\coordinate (E3) at (2,0);\coordinate (E4) at (2,-0.625);\coordinate (E5) at (2,-1.25);
\draw[->,color=red](D1)to[parallel segment above](D2);\draw[color=red] (D1)--(D2);\draw[->](D3)to[parallel segment above](T);\draw[->,below](T)to[parallel segment below](C3);\draw[->,below](C2)to[parallel segment below](C1);
\draw (D3)--(T)(T)--(E5)(E4)--(E3)(E2)--(E1)(S)--(A1)--(T)(S)--(B1)--(T)(S)--(C1)--(C2)--(C3)--(T)(S)--(D1)(D2)--(D3)(S)--(E1)(E2)--(E3)(E4)--(E5);
\draw[fill=white] (A1) circle [radius=3pt](B1) circle [radius=3pt](C1) circle [radius=3pt](C3) circle [radius=3pt](D2) circle [radius=3pt](E1) circle [radius=3pt](E3) circle [radius=3pt](E5) circle [radius=3pt];
\draw[fill=black](S) circle [radius=3pt](T) circle [radius=3pt](C2) circle [radius=3pt](D1) circle [radius=3pt](D3) circle [radius=3pt](E2) circle [radius=3pt](E4) circle [radius=3pt];
\node[below] at (T) {$t$};\node[above] at (S) {$s$};\node at (0, -3.5) {c.};
\end{tikzpicture}
\end{minipage}%
\begin{minipage}{0.25\textwidth}
\begin{tikzpicture}[scale=0.6]
\coordinate (S) at (0,2);\coordinate (T) at (0,-2);\coordinate (A1) at (-2,0);\coordinate (B1) at (-1,0);\coordinate (C1) at (0,1);\coordinate (C2) at (0,0);\coordinate (C3) at (0,-1);\coordinate (D1) at (1,1);\coordinate (D2) at (1,0);\coordinate (D3) at (1,-1);\coordinate (E1) at (2,1.25);\coordinate (E2) at (2,0.625);\coordinate (E3) at (2,0);\coordinate (E4) at (2,-0.625);\coordinate (E5) at (2,-1.25);
\draw[->,color=red](C2)to[parallel segment above](C1);\draw[color=red] (C2)--(C1);\draw[->](T)to[parallel segment below](C3);\draw[->](D3)to[parallel segment above](T);\draw[->](D1)to[parallel segment below](D2);
\draw(D3)--(T)(T)--(E5)(E4)--(E3)(D2)--(D1)(S)--(A1)--(T)(S)--(B1)--(T)(S)--(C1)(C2)--(C3)--(T)(S)--(D1)(D2)--(D3)(S)--(E1)--(E2)--(E3)(E4)--(E5);
\draw[fill=white] (A1) circle [radius=3pt](B1) circle [radius=3pt](C1) circle [radius=3pt](C3) circle [radius=3pt](D2) circle [radius=3pt](E1) circle [radius=3pt](E3) circle [radius=3pt](E5) circle [radius=3pt];
\draw[fill=black](S) circle [radius=3pt](T) circle [radius=3pt](C2) circle [radius=3pt](D1) circle [radius=3pt](D3) circle [radius=3pt](E2) circle [radius=3pt](E4) circle [radius=3pt];
\node[below] at (T) {$t$};\node[above] at (S) {$s$};\node at (0, -3.5) {d.};
\end{tikzpicture}
\end{minipage}
\caption{For each graph in the figure, the black vertices show a configuration of a minimum eternal vertex cover class, the red edge is the attacked one, and the arrows highlight the movement of the guards. Figures a. and b.: odd melon graph, the strategy described in the proof of~\Cref{lem:odd} according to the two cases of the proof of~\Cref{lem:elem}. Figures c. and d.: even melon graph, the two cases in the proof of~\Cref{lem:evenmelon}.}\label{fig:cboe} 
\end{figure}

\subsection{Mixed Melon Graphs}

To solve \EVC~on mixed melon graphs, we need two more definitions. 

Let $G$ be a mixed melon graph, $U\subset V$ a subset of vertices, and $P_o$ a path in $\po$, constituted by the sequence of vertices $v_0,\ldots,v_{2m+1}$, for some $m\geq 0$, such that $v_0=t$, $v_{2m+1}=s$. We distinguish the two following behaviors of $P_o$ with respect to $U$: we say that $P_o$ is an {\em \spath} with respect to $U$ (or simply \spath,~if $U$ is clear from the context) if $(U \cap V(P_o))\setminus \{t\}=\{v_{2i+1}~|~i\in \{0,\ldots,m\}$, while we say that $P_o$ is an {\em \tpath} with respect to $U$ (or simply \spath,~if $U$ is clear from the context) if $(U\cap V(P_o))\setminus \{s\}=\{v_{2i}~|~i\in \{0,\ldots,m\}\}$. Intuitively, in an \spath\ (respectively a \tpath) the vertices of $U$ alternate starting from $s$ (respectively $t$), regardless of whether $t$ (respectively $s$) is in $U$ or not. As an example, see~\Cref{fig:mixo}.a, that showcases both \spath s and a \tpath.

\begin{lemma}\label{le:ciclo_dispari}
Let $G$ be an odd $2$-melon graph with paths $P$ and $P'$, source $s$ and sink $t$. Moreover, let $U$ be a set of vertices such that $P$ is an \spath~and $P'$ is a \tpath~with respect to $U$ and let $U'$ be a set of vertices such that $P$ is a \tpath~and $P'$ is an \spath~with respect to $U$. Then it is possible to defend $G$ from an attack on any \sg~edge by shifting $U$ to $U'$ and vice versa.
\end{lemma}

\begin{proof}
Let $e=zw$ be an edge of $G$. Intuitively, to protect $e$, we move the guards to turn $P$ into a \tpath~and $P'$ into a \spath~following the direction of the forced shift of the guard on $e$. Let $e=zw$ be an edge of $G$. Since $C$ is a vertex cover and $e$ is \sg, it is not restrictive to assume that $z\in C$ and $w\notin C$. Call $u_0,\ldots,u_{2m+1}$ the vertices of $P$ and $v_0,\ldots,v_{2m'+1}$ the vertices of $P'$, for some $m, m' \geq 0$, and let $u_0=v_0=t$ and $u_{2m}=v_{2m'}=s$. Then, to defend from the attack on $e$, we move the guards to turn $P$ into a \tpath~and $P'$ into an \spath~following the forced shift of the guard from $z$ to $w$.

In particular, first, assume that $e$ is either an edge of $P$ and $z=u_{2j+1}$ and $w=u_{2j+2}$ for some $0 \leq j<m$, or an edge of $P'$ and $z=v_{2j+2}$ and $w=v_{2j+1}$ for some $1 \leq j<m'$. Then, to defend from the attack on $e$, we use the following defense function $\phi$:
\begin{itemize}
\item $\phi(u_{2i+1})=u_{2i+2}$ for every $i=0, \ldots , m-1$;
\item $\phi(v_{2i+2})=v_{2i+1}$ for $i=1, \ldots , m'-1$, ;
\item $\phi(s)=s$ and $\phi(t)=t$.
\end{itemize}
It is clear that $z$ shifts to $w$ and $C$ to $C'$. 

Now, assume that $e$ is either an edge of $P$ and $z=u_{2j+1}$ and $w=u_{2j}$ for some $0<j \leq m$, or an edge of $P'$ and $z=v_{2j}$ and $w=v_{2j+1}$ for some $0 \leq j \leq m'$. Then, to defend from the attack on $e$, we use the following defense function $\phi$:
\begin{itemize}
\item for every $i \leq m$, $\phi(u_{2i+1})=u_{2i}$;
\item for every $i\leq m'$, $\phi(v_{2i})=v_{2i+1}$.
\end{itemize}
\end{proof}

Let $P_e \in \pe$ and let $\mathcal{S}_o$ be any subset of $\po$. We denote with $\cf{P_e}{\mathcal{S}_o}$ the vertex set such that:
\begin{itemize}
\item $P_e$ is an \exte~w.r.t. $\cf{P_e}{\mathcal{S}_o}$;
\item every path in $\pe\setminus \{P_e\}$ is an \inte~w.r.t. $\cf{P_e}{\mathcal{S}_o}$;
\item every path in $\mathcal{S}_o$ is an \spath~w.r.t. $\cf{P_e}{\mathcal{S}_o}$;
\item every path in $\po\setminus \mathcal{S}_o$ is a \tpath~w.r.t. $\cf{P_e}{\mathcal{S}_o}$.
\end{itemize}

\begin{theorem}\label{lem:mixedmelonboth}
Let $G$ be a mixed $k$-melon graph, for some $k\geq 4$; if $|\pe|\geq 2$ and $|\po|\geq 2$, then it holds that $\e{G}=\m{G}+1$ and the family $\mathcal{U}=\{\cf{P_e}{\mathcal{S}_o}~|~P_e\in \pe,~\emptyset\neq \mathcal{S}_o\subset \po\}$ is a minimum eternal vertex cover class of $G$, where the sets $\cf{P_e}{\mathcal{S}_o}$ are defined above. \end{theorem}

\begin{proof}
For every path set $\mathcal{S}_o$ such that $\emptyset\neq \mathcal{S}_o\subset \po$, consider the set $U_{\mathcal{S}_o}$ of vertices of $G$ such that all the even paths are internal, the odd paths in $\mathcal{S}_o$ are $s$-paths and the remaining odd paths are $t$-paths. In other words, $U_{\mathcal{S}_o}$ differs from any $\cf{P_e}{\mathcal{S}_o}$ only in $P_e$ that is not external anymore, so $|\cf{P_e}{\mathcal{S}_o}|=|U_{\mathcal{S}_o}|+1$, for every $P_e\in \pe$. Moreover, let the family of the sets $U_{\mathcal{S}_o}$ be the collection of all minimum vertex covers of $G$; this is not an eternal vertex cover class of $G$ because it is not possible to defend from an attack on any edge that belongs to a path in $\pe$. It follows that $\e{G}$ is at least $\m{G}+1$ and hence proving that $\mathcal{U}$ is an eternal vertex cover class of $G$ also shows that $\mathcal{U}$ is minimum.

Let $\cf{P_e}{\mathcal{S}_o}$ be a configuration of $\mathcal{U}$ and let $e$ be an attacked \sg~edge of $G$. If $e$ is an edge of a path $P_e \in \pe$, let $G'$ be the subgraph of $G$ induced by the vertices of the paths in $\pe$. The definition of $\mathcal{U}$ implies that $G'$ contains at least an internal and at least an external path with respect to $\cf{P_e}{\mathcal{S}_o}\cap V(G')$, and we call $\phi'$ the defense function obtained from~\Cref{lem:evenmelon} when applied to the even melon graph $G'$ to protect it from the attack on $e$. Then, to protect $G$ from the attack on edge $e$ we define the defense function $\phi$ as follows: $\phi(v)=\phi'(v)$ if $v$ is a vertex of $G'$ and $\phi(v)=v$ otherwise. It is easy to see that $\phi$ protects $e$.

Let $P_o\in \po$ be the path that contains $e$ and $P'_o$ be another path of $\po$ such that $P_o\in \mathcal{S}_o$ if and only if $P'_o\in \po\setminus \mathcal{S}_o$ and consider the odd $2$-melon graph $G'$ induced by the vertices of the paths of $P_o$ and $P'_o$. We call $\phi'$ the defense function obtained from~\Cref{le:ciclo_dispari} when applied to the odd melon graph $G'$ to protect it from the attack on $e$. To protect $G$ from the attack on edge $e$ we define the defense function $\phi$ as follows: $\phi(v)=\phi'(v)$ if $v$ is a vertex of $G'$ and $\phi(v)=v$ otherwise (see~\Cref{fig:mixo}.b). It is easy to see that $\phi$ protects $e$.
\end{proof}

\begin{figure}
\centering
\begin{minipage}{0.4\textwidth}
\begin{tikzpicture}[scale=0.6]
\coordinate (S) at (0,2);\coordinate (T) at (0,-2);\coordinate (A1) at (-2.5,1);\coordinate (A2) at (-2.5,0);\coordinate (A3) at (-2.5,-1);\coordinate (B1) at (-1.5,1);\coordinate (B2) at (-1.5,0);\coordinate (B3) at (-1.5,-1);\coordinate (C1) at (-0.5,1);\coordinate (C2) at (-0.5,0);\coordinate (C3) at (-0.5,-1);\coordinate (D1) at (0.5,1.2);\coordinate (D2) at (0.5,0.4);\coordinate (D3) at (0.5,-0.4);\coordinate (D4) at (0.5,-1.2);\coordinate (E1) at (1.5,1.2);\coordinate (E2) at (1.5,0.4);\coordinate (E3) at (1.5,-0.4);\coordinate (E4) at (1.5,-1.2);\coordinate (F1) at (2.5,1.2);\coordinate (F2) at (2.5,0.4);\coordinate (F3) at (2.5,-0.4);\coordinate (F4) at (2.5,-1.2);\coordinate (SP) at (2.5,2);\coordinate (TP) at (2.5,-2);\coordinate (EX) at (-2.5,2);\coordinate (IN) at (-2.5,-2);
\draw(S)--(A1)--(A2)--(A3)--(T)(S)--(B1)--(B2)--(B3)--(T)(S)--(C1)--(C2)--(C3)--(T)(S)--(D1)--(D2)--(D3)--(D4)--(T)(S)--(E1)--(E2)--(E3)--(E4)--(T)(S)--(F1)--(F2)--(F3)--(F4)--(T);
\draw[dashed,color=red](2.5,-1.8)--(1.2,-1.4)(2.3,-1.8)--(0.4,-1.5)(2.5,1.8)--(2,1.4);\draw[dashed,color=blue](-2.3,-1.8)--(-1,-1.4)(-2.1,-1.8)--(-0.3,-1.5)(-2.5,1.8)--(-1.6,1.4);\node at (SP) {\tpath};
\node at (TP) {\spath s};\node at (EX) {\exte};\node at (IN) {\inte s};
\draw[fill=white] (A2) circle [radius=3pt](B1) circle [radius=3pt](B3) circle [radius=3pt](C1) circle [radius=3pt](C3) circle [radius=3pt](D1) circle [radius=3pt](D3) circle [radius=3pt](E1) circle [radius=3pt](E3) circle [radius=3pt](F1) circle [radius=3pt](F2) circle [radius=3pt](F3) circle [radius=3pt](F4) circle [radius=3pt];
\draw[fill=black](S) circle [radius=3pt](T) circle [radius=3pt](A1) circle [radius=3pt](A3) circle [radius=3pt](B2) circle [radius=3pt](C2) circle [radius=3pt](D2) circle [radius=3pt](D4) circle [radius=3pt](E2) circle [radius=3pt](E4) circle [radius=3pt](F1) circle [radius=3pt](F3) circle [radius=3pt];
\node[below] at (T) {$t$};\node[above] at (S) {$s$};\node at (0, -3.5) {a.};
\end{tikzpicture}
\end{minipage}%
\begin{minipage}{0.3\textwidth}
\begin{tikzpicture}[scale=0.6]
\coordinate (S) at (0,2);\coordinate (T) at (0,-2);\coordinate (A1) at (-2.5,1);\coordinate (A2) at (-2.5,0);\coordinate (A3) at (-2.5,-1);coordinate (B1) at (-1.5,1);\coordinate (B2) at (-1.5,0);\coordinate (B3) at (-1.5,-1);\coordinate (C1) at (-0.5,1);\coordinate (C2) at (-0.5,0);\coordinate (C3) at (-0.5,-1);\coordinate (D1) at (0.5,1.2);\coordinate (D2) at (0.5,0.4);\coordinate (D3) at (0.5,-0.4);\coordinate (D4) at (0.5,-1.2);\coordinate (E1) at (1.5,1.2);\coordinate (E2) at (1.5,0.4);\coordinate (E3) at (1.5,-0.4);\coordinate (E4) at (1.5,-1.2);\coordinate (F1) at (2.5,1.2);\coordinate (F2) at (2.5,0.4);\coordinate (F3) at (2.5,-0.4);\coordinate (F4) at (2.5,-1.2);\coordinate (SP) at (2.5,2);\coordinate (TP) at (2.5,-2);\coordinate (EX) at (-2.5,2);\coordinate (IN) at (-2.5,-2);
\draw[->,color=red](B2)to[parallel segment below](B1);\draw[color=red] (B2)--(B1);\draw[->](T)to[parallel segment above](B3);\draw[->](A3)to[parallel segment above](T);\draw[->](A1)to[parallel segment above](A2);
\draw(S)--(E1)(E4)--(T)(T)--(F4)(F3)--(F2)(F1)--(S)(S)--(A1)--(A2)--(A3)--(T)(S)--(B1)(B2)--(B3)--(T)(S)--(C1)--(C2)--(C3)--(T)(S)--(D1)--(D2)--(D3)--(D4)--(T)(E1)--(E2)--(E3)--(E4)(F1)--(F2)(F3)--(F4);
\draw[fill=white] (A2) circle [radius=3pt](B1) circle [radius=3pt](B3) circle [radius=3pt](C1) circle [radius=3pt](C3) circle [radius=3pt](D1) circle [radius=3pt](D3) circle [radius=3pt](E1) circle [radius=3pt](E3) circle [radius=3pt](F1) circle [radius=3pt](F2) circle [radius=3pt](F3) circle [radius=3pt](F4) circle [radius=3pt];
\draw[fill=black](S) circle [radius=3pt](T) circle [radius=3pt](A1) circle [radius=3pt](A3) circle [radius=3pt](B2) circle [radius=3pt](C2) circle [radius=3pt](D2) circle [radius=3pt](D4) circle [radius=3pt](E2) circle [radius=3pt](E4) circle [radius=3pt](F1) circle [radius=3pt](F3) circle [radius=3pt];
\node[below] at (T) {$t$};\node[above] at (S) {$s$};\node at (0, -3.5) {b.};
\end{tikzpicture}
\end{minipage}%
\begin{minipage}{0.3\textwidth}
\begin{tikzpicture}[scale=0.6]
\coordinate (S) at (0,2);\coordinate (T) at (0,-2);\coordinate (A1) at (-2.5,1);\coordinate (A2) at (-2.5,0);\coordinate (A3) at (-2.5,-1);\coordinate (B1) at (-1.5,1);\coordinate (B2) at (-1.5,0);\coordinate (B3) at (-1.5,-1);\coordinate (C1) at (-0.5,1);\coordinate (C2) at (-0.5,0);\coordinate (C3) at (-0.5,-1);\coordinate (D1) at (0.5,1.2);\coordinate (D2) at (0.5,0.4);\coordinate (D3) at (0.5,-0.4);\coordinate (D4) at (0.5,-1.2);\coordinate (E1) at (1.5,1.2);\coordinate (E2) at (1.5,0.4);\coordinate (E3) at (1.5,-0.4);\coordinate (E4) at (1.5,-1.2);\coordinate (F1) at (2.5,1.2);\coordinate (F2) at (2.5,0.4);\coordinate (F3) at (2.5,-0.4);\coordinate (F4) at (2.5,-1.2);\coordinate (SP) at (2.5,2);\coordinate (TP) at (2.5,-2);\coordinate (EX) at (-2.5,2);\coordinate (IN) at (-2.5,-2);
\draw[->,color=red](D2)to[parallel segment above](D3);\draw[color=red] (D2)--(D3);\draw[->](D4)to[parallel segment above](T);\draw[->]([shift={(0,-0.5)}]T.center)to[parallel segment below]([shift={(0,-0.5)}]F4.center);\draw[->](F3)to[parallel segment above](F2);\draw[->]([shift={(0,0.5)}]F1.center)to[parallel segment below]([shift={(0,0.5)}]S.center);\draw[->](S)to[parallel segment above](D1);
\draw(S)--(E1)(E4)--(T)(T)--(F4)(F3)--(F2)(F1)--(S)(S)--(A1)--(A2)--(A3)--(T)(S)--(B1)--(B2)--(B3)--(T)(S)--(C1)--(C2)--(C3)--(T)(S)--(D1)--(D2)(D3)--(D4)--(T)(E1)--(E2)--(E3)--(E4)(F1)--(F2)(F3)--(F4);
\draw[fill=white] (A2) circle [radius=3pt](B1) circle [radius=3pt](B3) circle [radius=3pt](C1) circle [radius=3pt](C3) circle [radius=3pt](D1) circle [radius=3pt](D3) circle [radius=3pt](E1) circle [radius=3pt](E3) circle [radius=3pt](F1) circle [radius=3pt](F2) circle [radius=3pt](F3) circle [radius=3pt](F4) circle [radius=3pt];
\draw[fill=black](S) circle [radius=3pt](T) circle [radius=3pt](A1) circle [radius=3pt](A3) circle [radius=3pt](B2) circle [radius=3pt](C2) circle [radius=3pt](D2) circle [radius=3pt](D4) circle [radius=3pt](E2) circle [radius=3pt](E4) circle [radius=3pt](F1) circle [radius=3pt](F3) circle [radius=3pt];
\node[below] at (T) {$t$};\node[above] at (S) {$s$};\node at (0, -3.5) {c.};
\end{tikzpicture}
\end{minipage}
\caption{For each graph in the figure, the black vertices show a configuration of a minimum eternal vertex cover class. Figures a., b. and c.: mixed melon graph with at least two even paths and two odd paths. Figure a. highlights even internal and external paths, and odd $s$- and $t$-paths. In Figures b. and c. the red edge is the attacked one, and the arrows highlight the movement of the guards, two cases in the proof of \Cref{lem:mixedmelonboth}.}\label{fig:mixo}  
\end{figure}
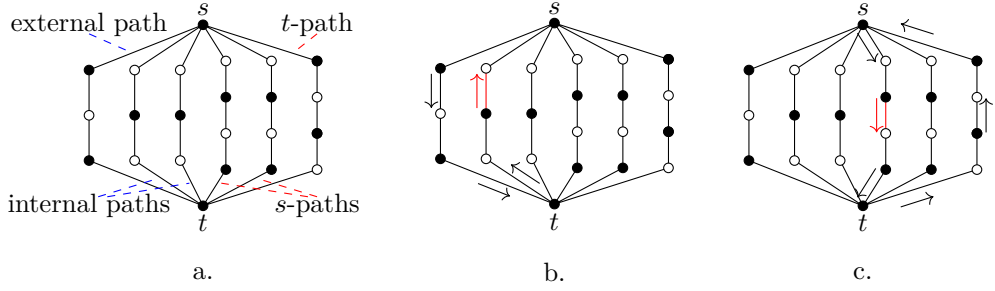

Consider now the case where $\po$ contains a single path $P_o$. Let $x\in \{s,t\}$ and $P_e\in \pe$. We denote with $\cf{x}{P_e}$ the vertex set such that:
\begin{itemize}
\item $P_e$ is an \exte~w.r.t. $\cf{x}{P_e}$;
\item every path in $\pe\setminus \{P_e\}$ is an \inte~w.r.t. $\cf{x}{P_e}$;
\item $P_o$ is an $x$-path w.r.t. $\cf{x}{P_e}$.
\end{itemize}

\begin{theorem}\label{lem:mixedmelonsingleo}
Let $G$ be a mixed $k$-melon graph, for some $k\geq 3$; if $|\po|=1$, then it holds that $\e{G}=\m{G}+1$ and the family $\,\mathcal{U}=\{\cf{x}{P_e}~|~x\in \{s,t\},P_e\in \pe\}$ is a minimum eternal vertex cover class of $G$, where the sets $\cf{x}{P_e}$ are defined above.
\end{theorem}

\begin{proof}
The graph $G$ has two minimum vertex covers $U_{x}$, $x\in \{s,t\}$: $U_{x}$ is the set of vertices of $G$ such that all even paths are \inte s and $P_o$ is a $x$-path. In other words, $U_{x}$ differs from any $\cf{x}{P_e}$ only in $P_e$ that is no longer external, so $|\cf{x}{P_e}|=|U_x|+1$, for every $P_e\in \pe$. 

Let $\cf{x}{P_e}$ be a configuration of $\mathcal{U}$. Due to the symmetry of $G$, it is not restrictive to assume that $x=s$. Let $e$ be an attacked edge.

If $e$ is an edge of a path in $\pe$, let $G'$ be the subgraph of $G$ induced by the vertices of the paths in $\pe$. The definition of $\mathcal{U}$ implies that $G'$ contains at least an internal and at least an external path, and we call $\phi'$ the defense function obtained from~\Cref{lem:evenmelon} when applied to the even melon graph $G'$ when protecting from the attack on $e$. To protect $G$, we define the defense function $\phi$ as follows: $\phi(v)=\phi'(v)$ if $v$ is a vertex of $G'$ and $\phi(v)=v$ otherwise. It is easy to see that $\phi$ protects $e$.

Suppose instead that $e=zw$ is an edge of the unique path $P_o\in \po$ and, without loss of generality, let $z\in \cf{s}{P_e}$ and $w\notin\cf{s}{P_e}$. Call $v_0,\ldots,v_{2m+1}$ the vertices of $P_o$, for some $m \geq 0$, $v_0=t$ and $v_{2m+1}=s$; since $P_o$ is an \spath, then $z=v_{2j+1}$ for some $j\leq m$. We distinguish two cases according to whether $w=v_{2j+2}$ or $w=v_{2j}$, that is, whether the guard on $z$ must be moved in the direction of $s$ or of $t$ in order to protect $e$.

If $w=v_{2j+2}$ (and hence $j<m$, see~\Cref{fig:mixeven}.a), we protect from the attack on edge $e$ by shifting all the guards on $P_o$ (except $t$) in the direction of $s$. Formally, the defense function $\phi$ is defined as follows: for every $0 \leq i<m$, $\phi(v_{2i+1})=v_{2i+2}$ and for every vertex $v$ of $G'$, $\phi(v)=v$. It is clear that $z$ shifts to $w$ and $\cf{s}{P_e}$ to $\cf{t}{P_e}$. 

If, instead, $w=v_{2j}$, for some $j\leq m$ (see~\Cref{fig:mixeven}.b), let $P'_e$ any even path of $G$ different from $P_e$. We have that $\cf{t}{P'_e}$ protects $\cf{s}{P_e}$ from the attack on $e$, shifting all the guards on $P_o$ in the direction of $t$ (and $P_o$ becoming a \tpath), and all the guards on $P'_{e}$ and on $P_e$ in the direction of $s$ ($P'_e$ becomes an \exte\ while $P_e$ becomes an \inte). In particular, say that the path $P_e$ has $\{u_0,\ldots,u_{2m_e}\}$ as vertices, with $v_0=t$ and $v_{2m_e}=s$, for some $m_e\geq 0$, and that the path $P_e'$ has $\{x_0,\ldots,x_{2m'_e}\}$ as vertices, with $x_0=t$ and $x_{2m'_e}=s$, for some $m'_e\geq 0$. Let $\phi$ be the defense function as follows:
\begin{itemize}
\item for every $0\leq i\leq m$, $\phi(v_{2i+1})=v_{2i}$; 
\item for every $0 \leq i<m_e$, $\phi(u_{2i+1})=u_{2i+2}$;
\item for every $0\leq i<m'_e$, $\phi(x_{2i})=x_{2i+1}$; 
\item for every vertex $u$, that is not part of neither $P_o$, $P_e$ nor $P_e'$, $\phi(u)=u$.
\end{itemize}
It is clear that $z$ shifts to $w$ and $\cf{s}{P_e}$ to $\cf{t}{P_e'}$. This completes the proof.
\end{proof}

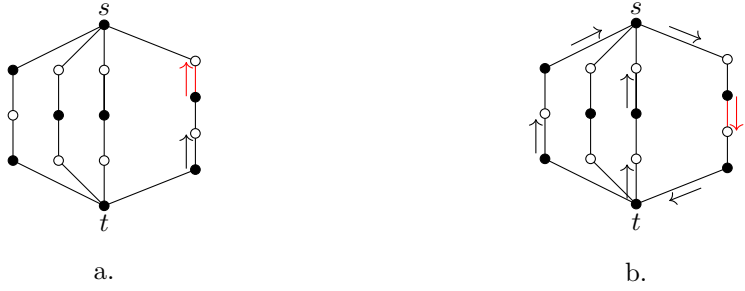
\begin{figure}
\begin{minipage}{0.5\textwidth}
\centering
\begin{tikzpicture}[scale=0.6]
\coordinate (S) at (0,2);\coordinate (T) at (0,-2);\coordinate (A1) at (-2,1);\coordinate (A2) at (-2,0);\coordinate (A3) at (-2,-1);\coordinate (B1) at (-1,1);\coordinate (B2) at (-1,0);\coordinate (B3) at (-1,-1);\coordinate (C1) at (0,1);\coordinate (C2) at (0,0);\coordinate (C3) at (0,-1);\coordinate (D1) at (2,1.2);\coordinate (D2) at (2,0.4);\coordinate (D3) at (2,-0.4);\coordinate (D4) at (2,-1.2);
\draw[->,color=red](D2)to[parallel segment below](D1);\draw[color=red] (D2)--(D1);\draw[->](D4)to[parallel segment below](D3);
\draw(S)--(D1)(D4)--(T)(T)--(C3)(C2)--(C1)(A3)--(A2)(A1)--(S)(A1)--(A2)(A3)--(T)(S)--(B1)--(B2)--(B3)--(T)(S)--(C1)--(C2)--(C3)(D2)--(D3)--(D4);
\draw[fill=white] (A2) circle [radius=3pt](B1) circle [radius=3pt](B3) circle [radius=3pt](C1) circle [radius=3pt](C3) circle [radius=3pt](D1) circle [radius=3pt](D3) circle [radius=3pt];
\draw[fill=black](S) circle [radius=3pt](T) circle [radius=3pt](A1) circle [radius=3pt](A3) circle [radius=3pt](B2) circle [radius=3pt](C2) circle [radius=3pt](D2) circle [radius=3pt](D4) circle [radius=3pt];
\node[below] at (T) {$t$};\node[above] at (S) {$s$};\node at (0, -3.5) {a.};
\end{tikzpicture}
\end{minipage}%
\begin{minipage}{0.5\textwidth}
\centering
\begin{tikzpicture}[scale=0.6]
\coordinate (S) at (0,2);\coordinate (T) at (0,-2);\coordinate (A1) at (-2,1);\coordinate (A2) at (-2,0);\coordinate (A3) at (-2,-1);\coordinate (B1) at (-1,1);\coordinate (B2) at (-1,0);\coordinate (B3) at (-1,-1);\coordinate (C1) at (0,1);\coordinate (C2) at (0,0);\coordinate (C3) at (0,-1);\coordinate (D1) at (2,1.2);\coordinate (D2) at (2,0.4);\coordinate (D3) at (2,-0.4);\coordinate (D4) at (2,-1.2);
\draw[->,color=red](D2)to[parallel segment below](D3);\draw[color=red] (D2)--(D3);\draw[->](D4)to[parallel segment below](T);\draw[->](S)to[parallel segment below](D1);\draw[->](T)to[parallel segment below](C3);\draw[->](C2)to[parallel segment below](C1);\draw[->](A3)to[parallel segment below](A2);\draw[->](A1)to[parallel segment below](S);
\draw(S)--(D1)(D4)--(T)(T)--(C3)(C2)--(C1)(A3)--(A2)(A1)--(S)(A1)--(A2)(A3)--(T)(S)--(B1)--(B2)--(B3)--(T)(S)--(C1)--(C2)--(C3)(D1)--(D2)(D3)--(D4);
\draw[fill=white] (A2) circle [radius=3pt](B1) circle [radius=3pt](B3) circle [radius=3pt](C1) circle [radius=3pt](C3) circle [radius=3pt](D1) circle [radius=3pt](D3) circle [radius=3pt];
\draw[fill=black](S) circle [radius=3pt](T) circle [radius=3pt](A1) circle [radius=3pt](A3) circle [radius=3pt](B2) circle [radius=3pt](C2) circle [radius=3pt](D2) circle [radius=3pt](D4) circle [radius=3pt];
\node[below] at (T) {$t$};\node[above] at (S) {$s$};\node at (0, -3.5) {b.};
\end{tikzpicture}
\end{minipage}
\caption{For each graph in the figure, the black vertices show a configuration of a minimum eternal vertex cover class, the red edge is the attacked one, and the arrows highlight the movement of the guards. Figures a. and b.: mixed melon graph with at least two even paths and only one odd path, the two cases in the proof of \Cref{lem:mixedmelonsingleo}.}\label{fig:mixeven}  
\end{figure}

Finally, consider the case where $\pe$ contains a single path $P_e$, and $\po$ contains at least two paths. The set $U_s$ (resp.\ $U_t$) is a vertex set not containing $t$ (resp.\ $s$) such that $P_e$ is an \exte~and every path in $\po$~is a \spath\ (resp. \tpath) w.r.t. $U_s$ (resp.\ $U_t$). Moreover, for any subset $\mathcal{S}_o$ of $\po$, let $U_{\mathcal{S}_o}$ be the vertex set of $G$ such that:
\begin{itemize}
\item $P_e$ is an \inte,
\item every path in $\mathcal{S}_o$ is a \spath
\item  every path in $\po\setminus \mathcal{S}_o$ is a \tpath~w.r.t. $U_{\mathcal{S}_o}$.
\end{itemize} 

Observe that $U_s$, $U_t$ and every $U_{\mathcal{S}_o}$ are vertex covers of $G$ and have all the same cardinality. Indeed, the extra guard present in the \exte\ is compensated by the presence of exactly one guard in $\{s,t\}$. Vice versa, the second guard on the set $\{s,t\}$ is compensated by one less guard in the \inte. As an example, see \Cref{fig:mixe}, that showcases the configurations $U_s$ and $U_{\mathcal{S}_o}$.

\begin{theorem}\label{lem:mixedmelonsinglee}
Let $G$ be a mixed $k$-melon graph, for some $k\geq 3$; if $|\pe|=1$, it holds that $\e{G}=\m{G}$ and the family $\mathcal{U}=\{U_s,U_t\}\cup \{U_{\mathcal{S}_o}~|~\emptyset \neq  \mathcal{S}_o\subset \po\}$ is a minimum eternal vertex cover class of $G$, where the sets $U_s$, $U_t$ and $U_{\mathcal{S}_o}$ are defined above. 
\end{theorem}

\begin{proof}
Every configuration of $\mathcal{U}$ is a minimum vertex cover of $G$; therefore, to prove the claim, it is enough to show that $\mathcal{U}$ is an eternal vertex cover class of $G$. Let $U$ be a configuration of $\mathcal{U}$ and $e$ be a \sg~edge of $G$. We consider two different cases, distinguishing whether $U$ is of the form either $U_x$, for some $x\in \{s,t\}$, or $U_{\mathcal{S}_o}$, for some non-empty proper subset $\mathcal{S}_o$ of $\po$.

\smallskip
\noindent
Case 1: $U$ is of the form $U_x$, for some $x\in \{s,t\}$. Thanks to the symmetry of $G$, it is not restrictive to assume $U=U_s$. For every non-empty proper subset $\mathcal{S}_o$ of $\po$, it holds that $U_{\mathcal{S}_o}$ protects $U_s$.

To prove this claim, it is not restrictive to assume that the attacked edge $e=zw$ is such that $z\in U_s$ and $w\not\in U_s$. We analyze different cases according to the position of $e$ in $G$. Let path $P_e$ be a sequence of vertices $u_0,\ldots,u_{2m}$, for some $m\geq 1$ such that $u_0=t$, $u_{2m}=s$. Moreover, let $P_o$ be any path in $\po$ and recall that $P_o$ is a \spath. Let $P_o$ be a sequence of vertices $v_0,\ldots,v_{2m_o+1}$, for some $m_o\geq 0$ such that $v_0=t$, $v_{2m_o+1}=s$.

Assume first that $e$ is an edge of the unique path $P_e\in \pe$, then $z=u_{2j+1}$ for some $j<m$. Informally, a defending strategy consists of turning the guards along the cycle formed by $P_e$ and any other odd path around it. Formally, the edge can be attacked in order to move the guard on $z$ either in the direction of $s$ or of $t$. In the first case, $w=u_{2j+2}$. To defend $G$ from this attack, we define a defense function $\phi$ as follows:
\begin{itemize}
\item for every $0 \leq i<m$, $\phi(u_{2i+1})=u_{2i+2}$;
\item for every $0 \leq i\leq m_o$, $\phi(v_{2i+1})=v_{2i}$;
\item for every $x$, that is not part of neither $P_e$ nor $P_o$, $\phi(v)=v$.
\end{itemize}

If, instead, the guard on $z$ is moved in the direction of $t$, then $w=u_{2j}$, for some $0 \leq j<m$. To defend $G$ from this attack, we define a defense function $\phi$ as follows:
\begin{itemize}
\item for every $0 \leq i<m$, $\phi(u_{2i+1})=u_{2i}$;
\item for every $0 \leq i<m_o$, $\phi(v_{2i+1})=v_{2i+2}$;
\item for every $v$, that is not part of neither $P_e$ nor $P_o$, $\phi(v)=v$.
\end{itemize}
It is easy to see that in both cases $\phi(z)=w$ and $\phi(U_s)=U_{\{P_o\}}$. 

Assume now that $e$ is an edge of some path $P_o\in \po$. It is easy to see that by exploiting one of the two defense functions defined above, we obtain to shift $U_s$ to $U_{\{P_o\}}$ and successfully defend from the attack on $e$. 

\smallskip
\noindent
Case 2: $U$ is of the form $U_{\mathcal{S}_o}$, for some non-empty proper subset $\mathcal{S}_o$ of $\po$. Then, either $U_x$,  with $x\in \{s,t\}$, or $U_{\mathcal{S}_o'}$, for some non-empty proper subset $\mathcal{S}_o'$ of $\po$, defends $U_{\mathcal{S}_o}$ from the attack on $e$. To prove this claim, assume again that $e=zw$ with $z\in U_{\mathcal{S}}$ and $w\not\in \mathcal{S}$. We analyze different cases according to the position of $e$ in $G$. 

First, suppose that $e$ is an edge of the unique path $P_e\in \pe$. Thanks to the symmetry of $G$, we can assume $z=u_{2j}$ and $w=u_{2j+1}$, for some $j<m$. To defend $G$ from this attack, we define a defense function $\phi$ as follows:
\begin{itemize}
\item for every $0\leq i<m$, $\phi(u_{2i})=u_{2i+1}$;
\item for every $x\neq t$ of $U_{\mathcal{S}_o}$ that is part of a \tpath, $\phi(x)=x_t$, where $x_t$ is the successor of $x$ in the path from $s$ to $y$ that contains $x$;
\item for every $x$ of $U_{\mathcal{S}_o}$ that is part of an \spath, $\phi(x)=x$.
\end{itemize}
It is easy to see that $\phi(z)=w$ and $\phi(U_{\mathcal{S}_o})=U_s$.

Now, assume that $e$ is an edge of some path $P_o\in \po$. Let $P_o'$ be another path of $\po$ such that $P_o'$ is a \tpath~if $P_o$ is an \spath~and an \spath~otherwise. Let $G'$ be the subgraph of $G$ induced by the vertices of the paths $P_o$ and $P_o'$. To defend from the attack on $e$ we define a defense function as follows: $\phi(v)=\phi'(v)$ if $v$ is a vertex of $G'$ and $\phi(v)=v$ otherwise, where $\phi'$ is the defense function obtained from~\Cref{le:ciclo_dispari} when applied to the odd $2$-melon graph $G'$ when defending from the attack on $e$. Finally, we assume that $e$ is an edge of some path $P_o\in \po$. Due to the symmetry of $G$, we can assume that $P_o$ is a \tpath. If $z=v_{2j+2}$ and $w=v_{2j+1}$, for some $j\leq m_o$, let $P_o'\in \po$ be a \spath~and define $\mathcal{S}_o'=(\mathcal{S}_o\setminus \{P_o'\})\cup \{P_o\}$. The path $P_o'$ is a sequence of vertices $w_0,\ldots,w_{2m_o'+1}$, for some $m_o'\geq 0$ such that $w_0=t$, $w_{2m_o'+1}=s$ and edges are of the form $w_iw_{i+1}$, for $i\leq 2m_o'$. To defend from the attack on $e$ we define a defense function $\phi$ as follows:
\begin{itemize}
\item for every $i<m_o$, $\phi(v_{2i+2})=v_{2i+1}$;
\item for every $i<m_o'$, $\phi(w_{2i+1})=w_{2i+2}$;
\item for every $x$ of $U_{\mathcal{S}_o}$ that is not part of neither $P_o$ or $P_o'$, $\phi(x)=x$.
\end{itemize}
It is easy to see that $\phi(z)=w$ and $\phi(U_{\mathcal{S}_o})=U_{\mathcal{S}_o'}$.

Finally, suppose that $u=v_{2h}$ and $v=v_{2h+1}$, for some $h\in \{0,\ldots,m_o\}$. Let $P_o'\in \po$ be a \spath~and define $\mathcal{S}_o'=(\mathcal{S}_o\setminus P_o')\cup P_o$. The path $P_o'$ is a sequence of vertices $w_0,\ldots,w_{2m_o'+1}$, for some $m_o'\geq 0$ such that $u_0=t$, $u_{2m_o'+1}=s$ and edges are of the form $u_hu_{h+1}$, for $h\in \{0,\ldots,2m_o'\}$. We have that the witness of $U_{\mathcal{S}_o}$ defending $U_{\mathcal{S}_o}$ from the attack on $e$ is given by the mapping $\phi$ defined as follows:
\begin{itemize}
\item for every $h\in \{0,\ldots,m_o\}$, $\phi(v_{2h})=v_{2h+1}$;
\item for every $h\in \{0,\ldots,m_o'\}$, $\phi(w_{2h+1})=w_{2h}$;
\item for every $z\in V(\mathcal{P_e}\setminus \{P_e,P_o'\}$), $\phi(z)=z$.
\end{itemize}
Again, it is easy to see that $\phi(u)=v$ and $\phi(U_{\mathcal{S}_o})=U_{\mathcal{S}_o'}$.
This completes the case analysis and the proof.
\end{proof}

\begin{figure}
\centering
\begin{minipage}{0.5\textwidth}
\begin{tikzpicture}[scale=0.6]
\coordinate (S) at (0,2);\coordinate (T) at (0,-2);\coordinate (A1) at (-2,1);\coordinate (A2) at (-2,0);\coordinate (A3) at (-2,-1);\coordinate (B1) at (0,1);\coordinate (B2) at (0,-1);\coordinate (C1) at (1,1);\coordinate (C2) at (1,-1);\coordinate (D1) at (2,1.2);\coordinate (D2) at (2,0.4);\coordinate (D3) at (2,-0.4);\coordinate (D4) at (2,-1.2);\coordinate (E1) at (3,1.2);\coordinate (E2) at (3,0.4);\coordinate (E3) at (3,-0.4);\coordinate (E4) at (3,-1.2);
\draw(S)--(A1)--(A2)--(A3)--(T)(S)--(B1)--(B2)--(T)(S)--(C1)--(C2)--(T)(S)--(D1)--(D2)--(D3)--(D4)--(T)(S)--(E1)--(E2)--(E3)--(E4)--(T);\draw[fill=white] (T) circle [radius=3pt](A2) circle [radius=3pt](B1) circle [radius=3pt](C1) circle [radius=3pt](D1) circle [radius=3pt](D3) circle [radius=3pt](E1) circle [radius=3pt](E3) circle [radius=3pt];
\draw[fill=black](S) circle [radius=3pt](A1) circle [radius=3pt](A3) circle [radius=3pt](B2) circle [radius=3pt](C2) circle [radius=3pt](D2) circle [radius=3pt](E2) circle [radius=3pt](D4) circle [radius=3pt](E4) circle [radius=3pt];
\node at (0, -3.5) {a.};
\end{tikzpicture}
\end{minipage}%
\begin{minipage}{0.5\textwidth}
\begin{tikzpicture}[scale=0.6]
\coordinate (S) at (0,2);\coordinate (T) at (0,-2);\coordinate (A1) at (-2,1);\coordinate (A2) at (-2,0);\coordinate (A3) at (-2,-1);\coordinate (B1) at (0,1);\coordinate (B2) at (0,-1);\coordinate (C1) at (1,1);\coordinate (C2) at (1,-1);\coordinate (D1) at (2,1.2);\coordinate (D2) at (2,0.4);\coordinate (D3) at (2,-0.4);\coordinate (D4) at (2,-1.2);\coordinate (E1) at (3,1.2);\coordinate (E2) at (3,0.4);\coordinate (E3) at (3,-0.4);\coordinate (E4) at (3,-1.2);
\draw(S)--(A1)--(A2)--(A3)--(T)(S)--(B1)--(B2)--(T)(S)--(C1)--(C2)--(T)(S)--(D1)--(D2)--(D3)--(D4)--(T)(S)--(E1)--(E2)--(E3)--(E4)--(T);
\draw[fill=white] (A1) circle [radius=3pt](A3) circle [radius=3pt](B1) circle [radius=3pt](B2) circle [radius=3pt](C2) circle [radius=3pt](D1) circle [radius=3pt](D3) circle [radius=3pt](E1) circle [radius=3pt](E3) circle [radius=3pt];
\draw[fill=black](S) circle [radius=3pt](T) circle [radius=3pt](A2) circle [radius=3pt](B1) circle [radius=3pt](C1) circle [radius=3pt](D2) circle [radius=3pt](D4) circle [radius=3pt](E2) circle [radius=3pt](E4) circle [radius=3pt];
\node at (0, -3.5) {b.};
\end{tikzpicture}
\end{minipage}
\caption{For each graph in the figure, the black vertices show a configuration of a minimum eternal vertex cover class. Figures a. and b.: mixed melon graph with at least two odd paths and only one even path, configurations $U_s$ and $U_{\mathcal{S}_o}$, respectively, used in the proof of \Cref{lem:mixedmelonsinglee}.}\label{fig:mixe}  
\end{figure}
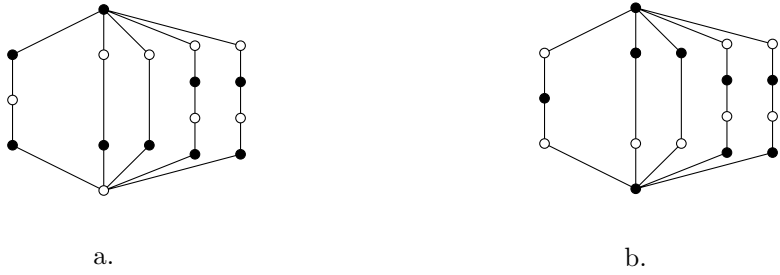

\subsection{Melon Graphs}

In the previous part of this section, we used the classification of melon graphs based on the parity of the paths constituting them to completely solve the \EVC\ problem on this graph class. We re-state the main result of this work, which summarizes the different cases providing a linear-time algorithm for \EVC~on melon graphs.

{\bf \Cref{thm:melons}.} {\em \EVC~is linear-time solvable for melon graphs.}

\begin{proof}
We start by running a BFS on $G$ rooted at its source $s$ to evaluate the cardinality $k$ of $\mathcal{P}$, and the two sets $\pe$ and $\po$. This takes $\mathcal{O}(|V|+|E|)$ time.

Recall that for $k\leq 2$, the claim is already well-known to be true. Hence, assume that $k\geq 3$. According to the cardinality of $\pe$ and $\po$, exactly one among Theorems \ref{lem:odd}, \ref{lem:evenmelon}, \ref{lem:mixedmelonboth}, \ref{lem:mixedmelonsingleo} and \ref{lem:mixedmelonsinglee} applies and the value of $\e{G}$ is obtained in constant time.\end{proof}

\section{Toward Eternal Vertex Cover on Series-Parallel Graphs}\label{sec:sp}

In view of the recursive structure of series-parallel graphs, it is natural to wonder whether it is possible to extend our main result of efficiently solving {\sc Minimum Eternal Vertex Cover} on melon graphs to the whole class of series-parallel graphs. We have reached the conclusion that this question has a negative answer, so we state the following conjecture:

\begin{con}
\EVC\ is NP-hard on series-parallel graphs.
\end{con}

This conjecture is based on many considerations, and the rest of this section is devoted to formalizing a couple of them. They show that melon graphs and \sep\ graphs behave differently w.r.t.\ the eternal vertex cover number and their SP-decompositions have different properties. These differences support our conjecture that computing the eternal vertex cover number on \sep\ graphs is significantly harder than computing the vertex cover number on \sep\ graphs or the eternal vertex cover number on melon graphs.

Preliminarily, it is worth noting that, given any graph $G$, it holds $\m{G} \leq \e{G} \leq 2 \m{G}$ \cite{KlostermeyerM09} and both the bounds are attainable: as an example, consider a cycle and an odd length path, respectively. Both these graphs are, in fact, \sep\ graphs, although rather special. In particular, paths (for which vertex cover and eternal vertex cover numbers are very far) are not biconnected; on the other hand, $k$-melon graphs, with $k\geq 2$, are biconnected, and vertex cover and eternal vertex cover numbers are either coinciding or very close. So, one could wonder whether the biconnectivity has some influence on the difference between these two parameters. The following result gives a negative answer to this question, showing that there are biconnected \sep\ graphs for which vertex cover and eternal vertex cover numbers are arbitrarily far in terms of difference and close to 2 in terms of ratio. As a side effect, this means that $2$ is the best approximation ratio in terms of $vc$ for biconnected \sep\ graphs.

\begin{lemma}\label{lem:unbounded}
For any integer $k\geq 0$, there is a biconnected \sep\ graph $G_{k}$ such that:
\begin{itemize}
\item $\e{G_k}-\m{G_k}\geq k$, and
\item $\e{G_k}\geq (2-\frac{2}{k-2})\m{G_k}$.
\end{itemize}
\end{lemma}

\begin{proof}
Let $H_k$ denote the $(k+3)$-melon graph where each of the $k+3$ paths is of length 2; in other words, $H_k$ is a complete bipartite graph $K_{2, k+3}$. Let $H_k'$ be the series composition of $H_k$ and of a 2-length path so that the source of $H_k'$ coincides with the source of the 2-length path and the sink of $H_k'$ coincides with the sink of $H_k$. For every $k\geq 2$, we define the biconnected series-parallel graph $G_k$ as the parallel composition of $k$ copies of $H_k'$ and one copy of $H_k$. Let $s_1, \ldots, s_k$, $s$ and $t$ be the sources of the $k$ copies of $H_k$ inside $H_k'$, the source of $G_k$ and the sink of $G_k$, respectively.  Note that $s$ and $t$ have a high degree, due to the presence of $H_k$, which is put in parallel with the copies of $H_k'$. See~\Cref{fig:ub} for a representation of $G_3$.

In order to show that $\e{G_k}$ and $\m{G_k}$ fulfill the inequalities of the claim, in the following, we first exactly evaluate $\m{G_k}$, then provide a lower bound for $\e{G_k}$.

Preliminarily, observe that $U=\{s_1,\ldots,s_k,s, t\}$ is the unique minimum vertex cover of $G_k$. Indeed, for any other vertex cover $U'\neq U$, if $U\subset U'$ then trivially $|U'|>|U|$, otherwise $U'$ does not contain $U$ and, for example, $s_i\notin U'$. 
This means that each of the $k+4$ neighbors of $s_i$ belongs to $U'$. Since the neighborhoods of each $s_j$ are disjoint, $|U'|\geq |U|+k+3=2k+5$. Even worse bounds are obtained when assuming that $s\notin U'$ or $t\notin U'$.  Thus, it holds that $\m{G_k}=k+2$.

Now, let $\mathcal{U}$ be a minimum eternal vertex cover class of $G_k$. Each configuration $U'$ of $\mathcal{U}$ must necessarily contain $U$ because, if by contradiction we supposed $U'$ does not include $U$, then we would have obtained $\e{G_k} = |U'| \geq 2k+5 > 2\m{G_k}$, which is absurd because $\e{G_k} \leq 2 \m{G_k}$~\cite{KlostermeyerM09}.

We exploit the property that $U \subset U'$, for each $U' \in \mathcal{U}$ to provide a lower bound for $\e{G_k}$. The informal idea is that guards on the vertices of $U$, which are the only vertices of $G_k$ having high degree, require an additional guard hosted by a neighboring vertex, so that they can be replaced to still defend $G_k$ whenever moved by the strategy. 

We now prove that every configuration $U' \in \mathcal{U}$ contains a vertex in $N[u]$ besides $u$, for each $u\in U$. If $N[u] \subseteq U'$, the claim is trivially true, so assume that there exists a neighbor $v$ of $u$ that is not in $U'$. Since $U'$ is a configuration of an eternal vertex cover class of $G_k$, there exists a defense function $\phi$ that protects $U'$ from the attack on $uv$ and, in particular, $\phi(u)=v$. 
Since $\phi(U')$, the configuration obtained from $U'$ after the defense, contains $U$ then it must exist a vertex $v'\in U'$ such that $\phi(v')=u$. Thus, $v'$ is a neighbor of $u$ that belongs to $U'$, which completes the proof of the claim. This means that $\e{G_k}\geq 2k+2$.

Thanks to the previous claim and to the fact that the $k$ sets $N[s_i]$ are pairwise disjoint, it holds that $|U'|\geq |U|+k$, that is $\e{G_k}-\m{G_k} \geq k$. Moreover, $\frac{\e{G_k}}{\m{G_k}}\geq \frac{2k+2}{k+2}=2-\frac{2}{k-2}$.
\end{proof}

\begin{figure}
\centering
\begin{tikzpicture}[scale=0.65]
\coordinate (A) at (-2,6);\coordinate (X1) at (-2.8,5);\coordinate (X2) at (-2,5);\coordinate (X3) at (-1.2,5);\coordinate (A1) at (-4,3);\coordinate (A2) at (-2,3);\coordinate (A3) at (0,3);\coordinate (B1) at (-7.25,1);\coordinate (B2) at (-6.75,1);\coordinate (B3) at (-6.25,1);\coordinate (B4) at (-5.75,1);\coordinate (B5) at (-5.25,1);\coordinate (B6) at (-4.75,1);\coordinate (B7) at (-3.25,1);\coordinate (B8) at (-2.75,1);\coordinate (B9) at (-2.25,1);\coordinate (B10) at (-1.75,1);\coordinate (B11) at (-1.25,1);\coordinate (B12) at (-0.75,1);\coordinate (B13) at (0.75,1);\coordinate (B14) at (1.25,1);\coordinate (B15) at (1.75,1);\coordinate (B16) at (2.25,1);\coordinate (B17) at (2.75,1);\coordinate (B18) at (3.25,1);\coordinate (D1) at (4.75,1);\coordinate (D2) at (5.25,1);\coordinate (D3) at (5.75,1);\coordinate (D4) at (6.25,1);\coordinate (D5) at (6.75,1);\coordinate (D6) at (7.25,1);\coordinate (C) at (0,-2);
\draw (A1)--(X1)--(A)--(X2)--(A2)(A3)--(X3)--(A) (A1)--(B1)--(C)--(B2)--(A1) (A1)--(B3)--(C)--(B4)--(A1) (A1)--(B5)--(C)--(B6)--(A1) (A2)--(B7)--(C)--(B8)--(A2) (A2)--(B9)--(C)--(B10)--(A2) (A2)--(B11)--(C)--(B12)--(A2) (A3)--(B13)--(C)--(B14)--(A3) (A3)--(B15)--(C)--(B16)--(A3) (A3)--(B17)--(C)--(B18)--(A3)(A)--(D1)--(C)--(D2)--(A)(A)--(D3)--(C)--(D4)--(A)(A)--(D5)--(C)--(D6)--(A);
\draw[fill=white] (A) circle [radius=3pt](X1) circle [radius=3pt](X2) circle [radius=3pt](X3) circle [radius=3pt](A1) circle [radius=3pt](A2) circle[radius=3pt](A3) circle [radius=3pt](B1) circle [radius=3pt](B2) circle [radius=3pt](B3) circle [radius=3pt](B4) circle [radius=3pt](B5) circle [radius=3pt](B6) circle [radius=3pt](B7) circle [radius=3pt](B8) circle [radius=3pt](B9) circle [radius=3pt](B10) circle [radius=3pt](B11) circle [radius=3pt](B12) circle [radius=3pt](B13) circle [radius=3pt](B14) circle [radius=3pt](B15) circle [radius=3pt](B16) circle [radius=3pt](B17) circle [radius=3pt](B18) circle [radius=3pt](C) circle [radius=3pt](D1) circle [radius=3pt](D2) circle [radius=3pt](D3) circle [radius=3pt](D4) circle [radius=3pt](D5) circle [radius=3pt](D6) circle [radius=3pt];
\draw[fill=black](A) circle [radius=3pt](A1) circle [radius=3pt](A2) circle [radius=3pt](A3) circle [radius=3pt](C) circle [radius=3pt];
\draw[fill=red](B1) circle [radius=3pt](B7) circle [radius=3pt](X3) circle [radius=3pt];
\node[above] at (A) {$s$};\node[above left] at (A1) {$s_1$};\node[above left] at (A2) {$s_2$};\node[above right] at (A3) {$s_3$};\node[below] at (C) {$t$};
\end{tikzpicture}
\caption{The figure shows the \sep\ graph $G_3$ described in the proof of~\Cref{lem:unbounded}. The black vertices represent its unique minimum vertex cover $U$. The red vertices are an example of the position of guards to be added to $U$ in order to get an eternal vertex cover configuration $U$.}\label{fig:ub}  
\end{figure}
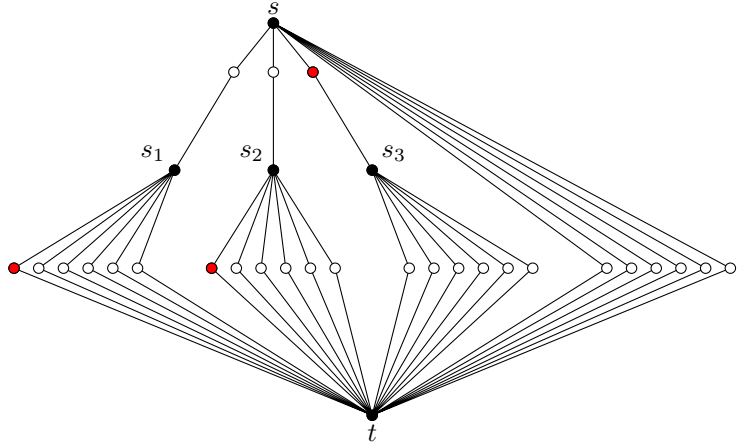

We propose a graph parameter that is well-defined on \sep\ graphs, which allows us to characterize melon graphs showing that they have a much simpler structure than general \sep\ graphs.

For a series-parallel graph $G$, we define the parameter $\al(G)$ as the maximum number of alternations between parallel and series nodes or {\em vice versa} in any path connecting the root and a leaf in any SP-decomposition of $G$. 

This parameter is clearly unbounded for the class of \sep\ graphs. The following result shows that melon graphs can be characterized as \sep\ graphs with $\al$ at most 1.

\begin{lemma}\label{lem:alt} 
For every melon graph $G$, then $\al(G)\leq 1$. Conversely, for every series-parallel graph $G$ with $\al(G)\leq 1$, either $G$ is a $k$-melon graph or is a path with possibly multiple edges.
\end{lemma}

\begin{proof}
First, let $G$ be a $k$-melon graph, for some $k\geq 1$, and let us prove by induction on $k$ that $\al(G) \leq 1$. If $k=1$, then $G$ is either a single edge or a path, which is obtained recursively by the series composition of two $1$-melon graphs. Thus, all non-leaf vertices of any SP-decomposition of $G$ are series vertices and then $\al(G)=0$. 

Suppose now $k\geq 2$. Then $G$ can only be obtained recursively by the parallel composition of two $x$- and $y$-melon graphs with $x,y\geq 1$ and $x+y=k$. Thus, every path $P$ connecting the root and a leaf in any SP-decomposition of $G$ starts with a non-empty sequence of parallel nodes and continues with a sequence of series nodes and so $P$ contains at most one alternation: $\al(G)\leq 1$.

Now, let $G$ be a series-parallel graph with $\al(G)\leq 1$ and fix any SP-decomposition $T$ of $G$. If $G$ is a single edge, then the statement trivially holds, so from now on we assume that $G$ has at least two edges. It is well known that the type of the root is the same in every SP-decomposition of $G$: indeed, the root is a series node if $G$ contains a cut-vertex and is a parallel node otherwise. If the root of $T$ is a parallel node, then $G$ is constituted by a set of parallel paths between two vertices, that is, $G$ is a melon graph. If the root of $T$ is a series node, then $G$ is a series of melon graphs in which the length of every path is one, {\em i.e.}, a set of multiple edges.
\end{proof}

Algorithmic techniques exploiting results on sub-structures, like divide and conquer or dynamic programming, look to be very natural on \sep\ graphs due to their recursive nature. Nevertheless, they do not immediately apply: while $\al\leq 1 $ for melon graphs guarantees a very limited number of cases, for the general case (where the \sep\ graph $G$ is constituted by either a series or a parallel composition of two \sep\ graphs $G_1$ and $G_2$), it is impractical to relate $\e{G}$ to $\e{G_1}$ and $\e{G_2}$.

The reason is that the defense strategies for the \EVC~problem are, in general, not local, that is, the defense against an attack may require that every guard of a given configuration to shift to a neighbor (see, for example, the strategy described in the proof of~\Cref{lem:odd}). The idea is that combining the local information about $G_1$ and $G_2$ graphs and elaborating such information to a global solution for $G$ is far from trivial.

\section{Conclusions}

The eternal vertex cover is a graph-theoretic representation of a 2-player game in rounds on a graph. Some vertices of this graph are occupied by so-called guards, who are able to cover all the edges that are incident to those vertices. The {\em attacker} is allowed to move one guard \textit{per} round along an edge with the goal of preventing the defender from winning. The {\em defender} replies by possibly moving the remaining guards along edges; it wins if it can make sure that, at every round of the game, all edges of the graph are covered. The task of the {\sc Minimum Eternal Vertex Cover} problem is to determine the minimum number of guards required by the defender to win. This problem has applications in network security where one aims to defend from a long series of malicious attacks.

The problem is known to be NP-hard in general. This paper fits in the research direction of understanding the structural and complexity properties of this problem when restricted to graph classes. We restrict our attention to the \sep\ graphs, a reasonably well-understood class for which many computationally hard problems become easy due to their recursive nature. 

We have shown that the {\sc Minimum Eternal Vertex Cover} problem can be solved in linear time for melon graphs: \sep\ graphs that are parallel composition of paths. This result is based on a case analysis of the structure of the input melon graph and generalizes the solution for cycles. Moreover, we have conjectured that this problem stays NP-hard on the whole class of \sep\ graphs. We have argued in favor of this conjecture exploiting the structural differences between melon and \sep\ graph based on the (eternal) vertex cover number and the SP-decomposition tree.

To further expand this work, we plan to consider the {\sc Minimum Eternal Vertex Cover} problem on outerplanar graphs, \textit{i.e.}, planar graphs that have a plane drawing with all vertices on the outer face. This class is interesting because, on the one hand, it is a subclass of \sep\ graphs and contains the maximal outerplanar graphs for which this problem is linear-time solvable~\cite{BKPW22}; on the other hand, the parameter $\al$ is unbounded for outerplanar graphs. We leave open whether a result similar to~\Cref{lem:unbounded} holds for this class.

\bibliographystyle{plainurl}
\bibliography{references}

\begin{thebibliography}{10}

\bibitem{AbrishamiCV22}
Tara Abrishami, Maria Chudnovsky, and Kristina Vuskovic.
\newblock Induced subgraphs and tree decompositions {I}. even-hole-free graphs
  of bounded degree.
\newblock {\em Journal of Combinatorial Theory, Series {B}}, 157:144--175,
  2022.
\newblock \href {https://doi.org/10.1016/J.JCTB.2022.05.009}
  {\path{doi:10.1016/J.JCTB.2022.05.009}}.

\bibitem{AN20}
Alessandro Aloisio and Alfredo Navarra.
\newblock Budgeted constrained coverage on series-parallel multi-interface
  networks.
\newblock {\em Internet Things}, 11:100259, 2020.
\newblock \href {https://doi.org/10.1016/J.IOT.2020.100259}
  {\path{doi:10.1016/J.IOT.2020.100259}}.

\bibitem{ASW23}
Rajeev Alur, Caleb Stanford, and Christopher Watson.
\newblock A robust theory of series parallel graphs.
\newblock {\em Proceedings of the ACM on Programming Languages},
  7(POPL):1058--1088, 2023.
\newblock \href {https://doi.org/10.1145/3571230} {\path{doi:10.1145/3571230}}.

\bibitem{AFI15}
Hisashi Araki, Toshihiro Fujito, and Shota Inoue.
\newblock On the eternal vertex cover numbers of generalized trees.
\newblock {\em {IEICE} Transactions on Fundamentals of Electronics,
  Communications and Computer Sciences}, 98-A(6):1153--1160, 2015.
\newblock \href {https://doi.org/10.1587/TRANSFUN.E98.A.1153}
  {\path{doi:10.1587/TRANSFUN.E98.A.1153}}.

\bibitem{AP89}
Stefan Arnborg and Andrzej Proskurowski.
\newblock Linear time algorithms for {NP}-hard problems restricted to partial
  k-trees.
\newblock {\em Discrete Applied Mathematics}, 23(1):11--24, 1989.
\newblock \href {https://doi.org/10.1016/0166-218X(89)90031-0}
  {\path{doi:10.1016/0166-218X(89)90031-0}}.

\bibitem{BabuCFPRW22}
Jasine Babu, L.~Sunil Chandran, Mathew~C. Francis, Veena Prabhakaran, Deepak
  Rajendraprasad, and J.~Nandini Warrier.
\newblock On graphs whose eternal vertex cover number and vertex cover number
  coincide.
\newblock {\em Discrete Applied Mathematics}, 319:171--182, 2022.
\newblock \href {https://doi.org/10.1016/J.DAM.2021.02.004}
  {\path{doi:10.1016/J.DAM.2021.02.004}}.

\bibitem{BKPW22}
Jasine Babu, K.~Murali Krishnan, Veena Prabhakaran, and J.~Nandini Warrier.
\newblock Eternal vertex cover number of maximal outerplanar graphs.
\newblock {\em arXiv}, 2022.

\bibitem{BabuMN22}
Jasine Babu, Neeldhara Misra, and Saraswati Nanoti.
\newblock Eternal vertex cover on bipartite graphs.
\newblock {\em Proc. {CSR} 2022}, 13296:64--76, 2022.
\newblock \href {https://doi.org/10.1007/978-3-031-09574-0\_5}
  {\path{doi:10.1007/978-3-031-09574-0\_5}}.

\bibitem{BabuP22}
Jasine Babu and Veena Prabhakaran.
\newblock A new lower bound for the eternal vertex cover number of graphs.
\newblock {\em Journal of Combinatorial Optimization}, 44(4):2482--2498, 2022.
\newblock \href {https://doi.org/10.1007/S10878-021-00764-8}
  {\path{doi:10.1007/S10878-021-00764-8}}.

\bibitem{BMM24}
Ishan Bansal, Ryan Mao, and Avhan Misra.
\newblock Network design on undirected series-parallel graphs.
\newblock {\em Proc. {ISCO} 2024}, 14594:277--288, 2024.
\newblock \href {https://doi.org/10.1007/978-3-031-60924-4\_21}
  {\path{doi:10.1007/978-3-031-60924-4\_21}}.

\bibitem{BGRR11}
Valentin Bonzom, Razvan Gurau, Aldo Riello, and Vincent Rivasseau.
\newblock Critical behavior of colored tensor models in the large {N} limit.
\newblock {\em Nuclear Physics B}, 853(1):174--195, 2011.

\bibitem{BMPY23}
Nick Brettell, Andrea Munaro, Dani{\"e}l Paulusma, and Shizhou Yang.
\newblock Comparing width parameters on graph classes.
\newblock {\em arXiv}, 2023.

\bibitem{DLP96}
Anders Dessmark, Andrzej Lingas, and Andrzej Proskurowski.
\newblock Faster algorithms for subgraph isomorphism of $k$-connected partial
  $k$-trees.
\newblock {\em Algorithmica}, 27(3):337--347, 2000.
\newblock \href {https://doi.org/10.1007/S004530010023}
  {\path{doi:10.1007/S004530010023}}.

\bibitem{D12}
Reinhard Diestel.
\newblock {\em Graph Theory, 4th Edition}, volume 173 of {\em Graduate texts in
  mathematics}.
\newblock Springer, 2012.

\bibitem{DissauxDNN23}
Thomas Dissaux, Guillaume Ducoffe, Nicolas Nisse, and Simon Nivelle.
\newblock Treelength of series-parallel graphs.
\newblock {\em Discrete Applied Mathematics}, 341:16--30, 2023.
\newblock \href {https://doi.org/10.1016/J.DAM.2023.07.022}
  {\path{doi:10.1016/J.DAM.2023.07.022}}.

\bibitem{Du65}
Richard~J Duffin.
\newblock Topology of series-parallel networks.
\newblock {\em Journal of Mathematical Analysis and Applications},
  10(2):303--318, 1965.

\bibitem{EGW01}
Wolfgang Espelage, Frank Gurski, and Egon Wanke.
\newblock How to solve {NP}-hard graph problems on clique-width bounded graphs
  in polynomial time.
\newblock {\em Proc. {WG} 2001}, 2204:117--128, 2001.
\newblock \href {https://doi.org/10.1007/3-540-45477-2\_12}
  {\path{doi:10.1007/3-540-45477-2\_12}}.

\bibitem{FominGGKS10}
Fedor~V. Fomin, Serge Gaspers, Petr~A. Golovach, Dieter Kratsch, and Saket
  Saurabh.
\newblock Parameterized algorithm for eternal vertex cover.
\newblock {\em Information Processing Letters}, 110(16):702--706, 2010.
\newblock \href {https://doi.org/10.1016/J.IPL.2010.05.029}
  {\path{doi:10.1016/J.IPL.2010.05.029}}.

\bibitem{GHO11}
Robert Ganian, Petr Hlinen{\'{y}}, and Jan Obdrz{\'{a}}lek.
\newblock Clique-width: When hard does not mean impossible.
\newblock {\em Proc. {STACS} 2011}, 9:404--415, 2011.
\newblock \href {https://doi.org/10.4230/LIPICS.STACS.2011.404}
  {\path{doi:10.4230/LIPICS.STACS.2011.404}}.

\bibitem{GS15}
Martin Grohe and Pascal Schweitzer.
\newblock Isomorphism testing for graphs of bounded rank width.
\newblock {\em Proc. {FOCS} 2015}, pages 1010--1029, 2015.
\newblock \href {https://doi.org/10.1109/FOCS.2015.66}
  {\path{doi:10.1109/FOCS.2015.66}}.

\bibitem{GN96}
Arvind Gupta and Naomi Nishimura.
\newblock The complexity of subgraph isomorphism for classes of partial
  $k$-trees.
\newblock {\em Theoretical Computer Science}, 164(1-2):287--298, 1996.
\newblock \href {https://doi.org/10.1016/0304-3975(96)00046-1}
  {\path{doi:10.1016/0304-3975(96)00046-1}}.

\bibitem{Ga64}
G{\'a}bor Hetyei.
\newblock Rectangular configurations which can be covered by 2$\times$1
  rectangles.
\newblock {\em P{\'e}csi Tan. Foisk. K{\"o}zl}, 8:351--367, 1964.

\bibitem{KlostermeyerM09}
William~F. Klostermeyer and Christina~M. Mynhardt.
\newblock Edge protection in graphs.
\newblock {\em Australasian Journal of Combinatorics}, 45:235--250, 2009.

\bibitem{KMC16}
William~F Klostermeyer and Christina~M Mynhardt.
\newblock Protecting a graph with mobile guards.
\newblock {\em Applicable Analysis and Discrete Mathematics}, 10(1):1--29,
  2016.

\bibitem{KKM18}
Nils Kriege, Florian Kurpicz, and Petra Mutzel.
\newblock On maximum common subgraph problems in series–parallel graphs.
\newblock {\em European Journal of Combinatorics}, 68:79--95, 2018.
\newblock \href {https://doi.org/10.1016/j.ejc.2017.07.012}
  {\path{doi:10.1016/j.ejc.2017.07.012}}.

\bibitem{LPPS14}
Daniel Lokshtanov, Marcin Pilipczuk, Michal Pilipczuk, and Saket Saurabh.
\newblock Fixed-parameter tractable canonization and isomorphism test for
  graphs of bounded treewidth.
\newblock {\em {SIAM} Journal on Computing}, 46(1):161--189, 2017.
\newblock \href {https://doi.org/10.1137/140999980}
  {\path{doi:10.1137/140999980}}.

\bibitem{Ma19}
Hosam~M. Mahmoud.
\newblock A spectrum of series-parallel graphs with multiple edge evolution.
\newblock {\em Probability in the Engineering and Informational Sciences},
  33(4):487–499, 2019.
\newblock \href {https://doi.org/10.1017/S0269964818000505}
  {\path{doi:10.1017/S0269964818000505}}.

\bibitem{MT92}
Ji{\v{r}}{\'\i} Matou{\v{s}}ek and Robin Thomas.
\newblock On the complexity of finding iso-and other morphisms for partial
  $k$-trees.
\newblock {\em Discrete Mathematics}, 108(1-3):343--364, 1992.
\newblock \href {https://doi.org/10.1016/0012-365X(92)90687-B}
  {\path{doi:10.1016/0012-365X(92)90687-B}}.

\bibitem{MisraN22}
Neeldhara Misra and Saraswati Nanoti.
\newblock Eternal vertex cover on bipartite and co-bipartite graphs.
\newblock {\em CoRR}, 2022.

\bibitem{MisraN23}
Neeldhara Misra and Saraswati~Girish Nanoti.
\newblock Spartan bipartite graphs are essentially elementary.
\newblock {\em Proc. {MFCS} 2023}, 272(68):1--15, 2023.
\newblock \href {https://doi.org/10.4230/LIPICS.MFCS.2023.68}
  {\path{doi:10.4230/LIPICS.MFCS.2023.68}}.

\bibitem{PP24}
Kaustav Paul and Arti Pandey.
\newblock Some algorithmic results for eternal vertex cover problem in graphs.
\newblock {\em Journal of Graph Algorithms and Applications}, 28(3):69--85,
  2024.
\newblock \href {https://doi.org/10.7155/JGAA.V28I3.2972}
  {\path{doi:10.7155/JGAA.V28I3.2972}}.

\bibitem{SintiariT21}
Ni~Luh~Dewi Sintiari and Nicolas Trotignon.
\newblock ({T}heta, triangle)-free and (even hole, $k_4$))-free graphs - part
  1: Layered wheels.
\newblock {\em Journal Graph Theory}, 97(4):475--509, 2021.
\newblock \href {https://doi.org/10.1002/JGT.22666}
  {\path{doi:10.1002/JGT.22666}}.

\bibitem{Sy83}
M~Syslo.
\newblock {NP}-complete problems on some tree-structured graphs: a review.
\newblock {\em Proc. {WG} 1983}, 1983.

\bibitem{tollis89}
Ioannis~G Tollis.
\newblock On finding a minimum vertex cover of a series-parallel graph.
\newblock {\em Applied Mathematics Letters}, 2(3):305--309, 1989.

\bibitem{ValdesTL82}
Jacobo Valdes, Robert~E. Tarjan, and Eugene~L. Lawler.
\newblock The recognition of series parallel digraphs.
\newblock {\em SIAM Journal on Computing}, 11(2):298--313, 1982.
\newblock \href {https://doi.org/10.1137/0211023} {\path{doi:10.1137/0211023}}.

\bibitem{Va12}
Martin Vatshelle.
\newblock New width parameters of graphs.
\newblock {\em The University of Bergen}, 2012.

\bibitem{ZMN05}
Xiao Zhou, Yuki Matsuo, and Takao Nishizeki.
\newblock List total colorings of series-parallel graphs.
\newblock {\em J. Discrete Algorithms}, 3(1):47--60, 2005.
\newblock \href {https://doi.org/10.1016/J.JDA.2003.12.006}
  {\path{doi:10.1016/J.JDA.2003.12.006}}.

\bibitem{ZTN00}
Xiao Zhou, Syurei Tamura, and Takao Nishizeki.
\newblock Finding edge-disjoint paths in partial \emph{k}-trees.
\newblock {\em Algorithmica}, 26(1):3--30, 2000.
\newblock \href {https://doi.org/10.1007/S004539910002}
  {\path{doi:10.1007/S004539910002}}.

\end{thebibliography}
\end{document}